\documentclass[review]{elsarticle}
\usepackage{epstopdf}
\usepackage{amssymb}
\usepackage{amsmath}
\usepackage{amsthm}
\usepackage{amsbsy}
\usepackage{psfrag}
\usepackage{pstricks}
\usepackage{graphics}
\usepackage{graphicx}
\usepackage{graphicx,amsmath,amsthm,wasysym,amssymb,mathrsfs}
\usepackage{mathtools,appendix,epsfig,epsf,color,subfigure}
\usepackage{hyperref}

\journal{}
\theoremstyle{plain}
\newtheorem{theorem}{Theorem}
\newtheorem{lemma}[theorem]{Lemma}

\newtheorem{example}{Example}
\theoremstyle{remark}

\begin{document}
\allowdisplaybreaks[4]
%
%
%

\title {Robust Numerical Solution for Solving Elastohydrodynamic Lubrication (EHL) Problems 
using Total Variation Diminishing (TVD) Approach}
%
\author[add1]{Peeyush Singh\corref{cor1}}
\ead{peeyushs8@gmail.com,peeyush@tifrbng.res.in}
\cortext[cor1]{Corresponding author}

\address[add1]{Tata Institute of Fundamental Research Centre for Applicable Mathematics,Bangalore-560 065,India}

\date{}
\begin{abstract}
In this study, we propose a class of total variation diminishing (TVD) schemes for solving pseudo-monotone variational 
inequality arises in elasto-hydrodynamic lubrication point contact problem.
A limiter based stable hybrid line splittings are introduced on hierarchical multi-level grid.
These hybrid splittings are designed by use of diffusive coefficient 
and mesh dependent switching parameter in the computing domain of interest. 
The spectrum of illustrated splittings is derived with the help of well known local Fourier analysis (LFA). 
Numerical tests validate the performance of scheme and its competitiveness to the previous existing schemes. 
Advantages of proposed splittings are observed in the sense that it reduces computational complexity (up to ($O(n\log n)$)
and solve high order discretization directly (no defect-correction tool require) without perturbing
the robustness of the solution procedure (i.e. it works well for large range of load parameters).
\end{abstract}
\begin{keyword}
TVD schemes \sep Defect-correction \sep multi-grid
\sep Elastohydrodynamic Lubrication contact problem \sep variational inequalities
\MSC  65N06 \sep 65N55 \sep 65K15 \sep 35R35 \sep 45K05
\end{keyword}

\maketitle
\allowdisplaybreaks
\def\R{\mathbb{R}}
\def\cA{\mathcal{A}}
\def\cK{\mathcal{K}}
\def\cN{\mathcal{N}}
\def\p{\partial}
\def\O{\Omega}
\def\bbP{\mathbb{P}}
\def\cV{\mathcal{V}}
\def\cM{\mathcal{M}}
\def\cT{\mathcal{T}}
\def\cE{\mathcal{E}}
\def\bF{\mathbb{F}}
\def\bC{\mathbb{C}}
\def\bN{\mathbb{N}}
\def\ssT{{\scriptscriptstyle T}}
\def\HT{{H^2(\O,\cT_h)}}
\def\mean#1{\left\{\hskip -5pt\left\{#1\right\}\hskip -5pt\right\}}
\def\jump#1{\left[\hskip -3.5pt\left[#1\right]\hskip -3.5pt\right]}
\def\smean#1{\{\hskip -3pt\{#1\}\hskip -3pt\}}
\def\sjump#1{[\hskip -1.5pt[#1]\hskip -1.5pt]}
\def\jumptwo{\jump{\frac{\p^2 u_h}{\p n^2}}}

\section{Introduction}\label{sec:one}
Elasto-hydrodynamic lubrication (EHL) is more often understood as a phenomenon of fluid film lubrication in which the natural 
process of hydrodynamic fluid film creation is governed due to deformation of contacting bodies and lubricant viscosity increases 
due to high pressure.
Significant contributions have been made by many researchers in the development of more efficient and accurate methods 
for the study of EHL in last few decades (e.g.\cite{v1,vr,hamrock,ehlbook,cimatti,ALubrecht,sahmed14,peeyush,moes,venner94}).
It is well known that many numerical solutions of EHL model suffer lack of numerical stability and convergence 
during computation, if not tackled correctly. On the other hand, any direct solver such as Newton-Raphson technique takes a lot of 
computational storage and time (up to $O({n^{3}})$) to solve the model and hence it has no commercial use in practice.
In 1992, Venner \cite{vr} has introduced a low order discretization for EHL model (see ~\ref{model:ehl}) which is stable for larger range of load parameters.
However, author's best knowledge stable schemes for the EHL model~\ref{model:ehl} are largely unavailable in literature which work well for very large
range of load parameters other than Venner approach and in that sense it turns out to be a challenging problem in scientific community.
The main numerical difficulty in these problems occurs due to lack of stable smoother and poor approximation of pressure 
profile near its steep gradient location by any standard iterative procedure.
Also when applied load in contacting bodies are sufficiently high then many people observed wiggles in pressure 
and film thickness profile by using central or any high order scheme in convection term of Reynolds equation.
One possible way to overcome the difficulty, people have used lower order discretization in convection term. 
In addition, for obtaining the high order stable, accurate solutions for such problems, researchers have 
applied lower order scheme in a defect corrected way through a suitable higher order discretization. 
However, such defect-correction \cite{koren,koren88} setting most the time is not able to solve the difficulty in the sense that it
does not reduce residual accurately due to poor conditioning of matrix in outer iteration (e.g.\cite{Oosterlee}).
Furthermore, lower order schemes are more diffusive and allow to produce smoothing effect in the steep gradient region of solution
and less accurate in the smooth part of the solution.
This is the main motivation for present study to adopt total variation diminishing (TVD) 
approach for the EHL model problem. The reason behind TVD schemes for EHL model have been rarely applied so far in literature
due to the fact that implementation is not obvious and straight forward as the case of linear-convection diffusion due 
to strong coupling of pressure and film thickness term in existing model.
Therefore, in this article an attempt has been made to solve the problem generalizing TVD concept efficiently in the existing EHL model.\\
TVD schemes are understood as a generalized form of upwind based discretized schemes (more detailed definition will define later).
Mostly, such schemes have been extensively devised for solving time dependent gas dynamics problems. Later on people have started to apply such concept for
steady state problem in many CFD applications. Initially, the concept of TVD has been established by Harten 
and later by Sweby \cite{harten83,harten84,sweby} to avoid unphysical wiggles in a numerical scheme.
Harten also has given necessary and sufficient condition for a scheme to be TVD. To understand the concept, we first define 
the notation total variation $TV$ of a mesh function $u^{n}$ as
\begin{align}
\label{eqn1}
 TV(u^{n}) = \displaystyle\sum_{-\infty}^{\infty}|u_{j+1}^{n}-u_{j}^{n}|=\displaystyle\sum_{-\infty}^{\infty}|\Delta_{j+1/2}u^{n}|
\end{align}
having the following convention 
\begin{align}
\label{eqn2}
 \Delta_{j+1/2}u^{n} = u_{j+1}^{n}-u_{j}^{n}
\end{align}
for any mesh function $u$ is used.
Harten's theory is understood in the form of conservation laws 
\begin{align}
\label{eqn3}
 u_{t}+ f(u)_{x} = 0.
\end{align}
The numerical approximation of Eq.~(\ref{eqn3}) is said to be TVD if
\begin{align}
\label{eqn4}
 TV(u^{n+1}) \le TV(u^{n})
\end{align}
Then Harten's condition for any scheme to be TVD is explained below.
\begin{theorem}Let a general numerical scheme for conservation laws Eq.~(\ref{eqn3}) is of the form
 \begin{gather}
 \label{eqn5}
 u^{n+1}_{i}=u^{n}_{i}-c_{i}^{n}(u_{i}^{n}-u_{i-1}^{n})+d_{i}^{n}(u_{i+1}^{n}-u_{i}^{n})
 \end{gather}
 over one time step, where the coefficients $c_{i}^{n}$ and $d_{i}^{n}$ are arbitrary value (In
 practice it may depend on values $u^{n}_{i}$ in some way i.e., the method may be nonlinear).
 Then $TV(u^{n+1}) \leq TV(u^{n})$ provided  the following conditions are satisfied
 \begin{gather}
 \label{eqn6}
 c^{n}_{i} \geq 0 \quad ,
 d^{n}_{i} \geq 0\quad ,
 c^{n}_{i}+d^{n}_{i} \leq 1\quad \forall i
  \end{gather}
 \end{theorem}
There has been a very well developed TVD theory available in literature for time dependent problem.
Additionally, this concept is also extended for steady state convection-diffusion case in the form of $M$- matrix \cite{Varga} 
using appropriate flux limiting schemes \cite{Oosterlee,koren,koren88,osterlee2003}. However, very little attention have been paid
in developing TVD schemes for EHL problems. In this article, our aim to investigate a class of splitting for EHL model which is robust and
high order accurate ( at least second order in smooth part of the solution ) for larger range of load parameters.
\subsection{Model Problem}\label{model:ehl}
The following two dimensional circular point contact model problem is 
taken for numerical study defined below in the form of variational inequality 
written in non dimensional form
\begin{align}
\label{eqn7}
 \frac{\partial }{\partial x} \Big(\epsilon \frac {\partial  u}{\partial x}\Big)+
 \frac{\partial }{\partial y} \Big(\epsilon \frac {\partial  u}{\partial y}\Big)
 \le \frac {\partial (\rho \mathscr{H})}{\partial x} \quad \in \quad \Omega \nonumber \\
 u\ge 0 \quad \in \quad \Omega \nonumber   \\
 u.\Big[\frac{\partial }{\partial x} \Big(\epsilon \frac {\partial  u}{\partial x}\Big)+
 \frac{\partial }{\partial y} \Big(\epsilon \frac {\partial  u}{\partial y}\Big)
 -\frac {\partial (\rho \mathscr{H})}{\partial x}\Big] = 0 \quad  \in \quad \Omega,
\end{align}
where $u$ is non-dimensional pressure of liquid (lubricant) and $\Omega$ is sufficiently
large bounded domain such that
 \begin{align}
  \label{eqn8}
  u= 0 \quad \text{on} \quad \partial \Omega.
 \end{align}
 Here term $\epsilon$  is defined as
\begin{equation*}
 \epsilon = \frac{\rho \mathscr{H}^{3}}{\eta\lambda},
\end{equation*}
where $\rho$ is dimensionless density of lubrication, $\eta$ is dimensionless viscosity of lubrication and
speed parameter
\begin{align}
  \label{eqn9}
   \lambda= \dfrac{6\eta_{0}u_{s}R^{2}}{a^{3}p_{H}}.
  \end{align}
The non-dimensionless viscosity $\eta$ is defined according to 
\begin{align}
\label{eqn10}
 \eta(u) = \exp\Bigg\{ \Bigg( \dfrac{\alpha p_{0}}{z}  \Bigg) 
 \Bigg(-1+\Big(1+\dfrac{{u}p_{H}}{p_{0}}\Big)^{z}   \Bigg)   \Bigg\}.
\end{align}
Dimensionless density $\rho$ is given by
\begin{align}
\label{eqn11}
 \rho(u) = \dfrac{0.59 \times 10^{9} + 1.34 u p_{H}}{0.59 \times 10^{9} + u p_{H}}.
\end{align}
The term film thickness $\mathscr{H}$ of lubricant is written as follows
\begin{align}
\label{eqn12}
\mathscr{H}(x,y) = \mathscr{H}_{00}+\frac{x^{2}}{2}+\frac{y^{2}}{2} + 
\frac{2}{\pi^{2}}\int_{-\infty}^{\infty} \int_{-\infty}^{\infty}\frac{u(x^{'},y^{'})dx^{'}dy^{'}}{\sqrt{(x-x^{'})^2+(y-y^{'})^2}},
\end{align}
where $\mathscr{H}_{00}$ is an integration constant.\\
The dimensionless force balance equation is defined as follows
\begin{gather}
\label{eqn13}
 \int_{-\infty}^{\infty} \int_{-\infty}^{\infty}u(x',y') dx'dy' = \frac{3\pi}{2}
\end{gather}
All notations used in EHL model are defined in ~\ref{app:one}.\\
     \begin{figure}
       \centering
        \includegraphics[width=8cm,height=10cm,keepaspectratio]{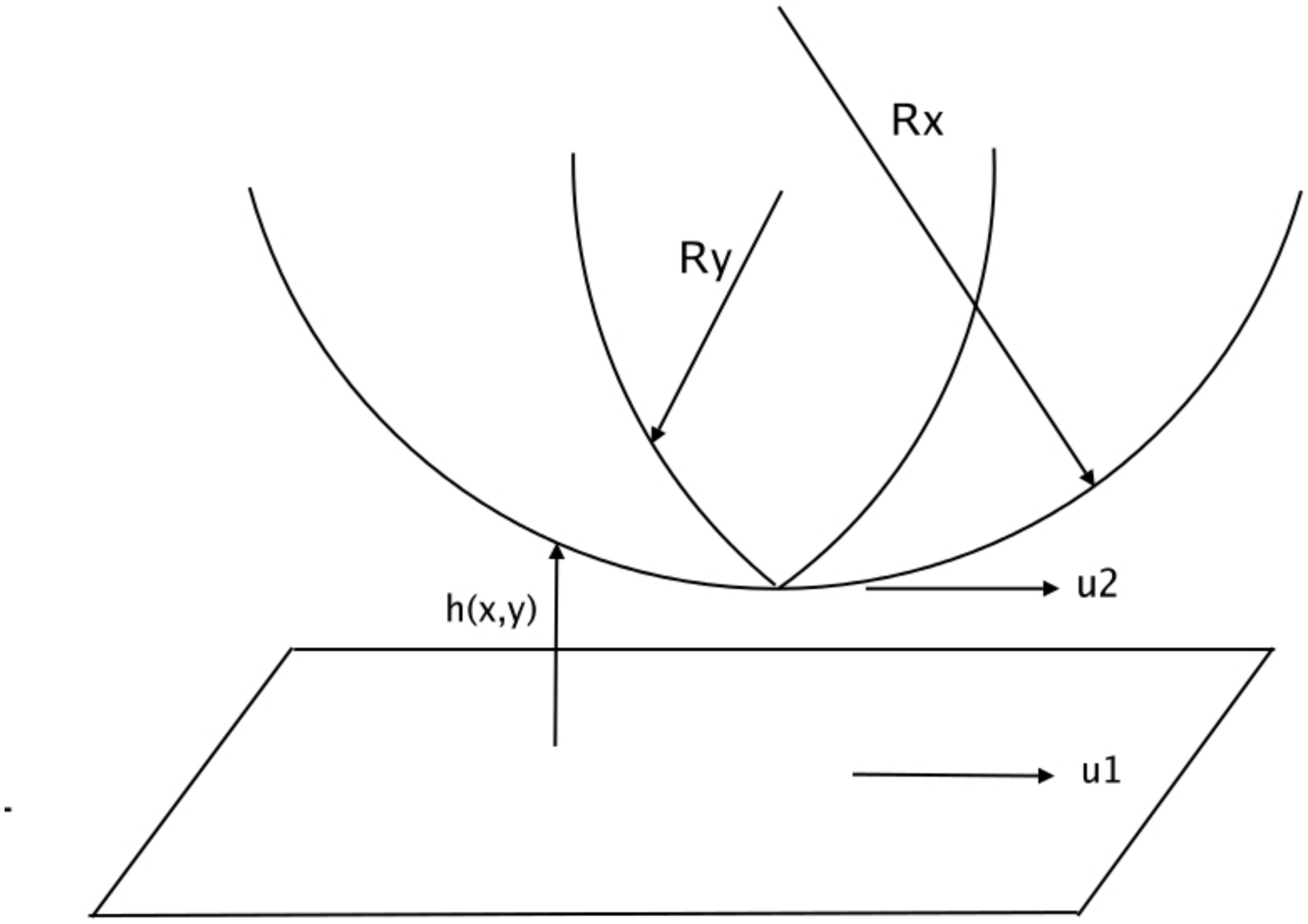}
        \caption{Schematic diagram of EHL point contact model}
        \label{fig:fig1}
        \end{figure}       
A schematic diagram of EHL point contact model is given in Fig.~\ref{fig:fig1}.     
Rest of the article is organized as followed. In Section.~\ref{section:s2}, few preliminaries are discussed which require 
in numerical study of EHL model which help in subsequent numerical analysis of the model.
In Section~\ref{sec:three}, a series of splitting are constructed by imitating linear convection-diffusion model and linear EHL model. 
In Section~\ref{sec:four}, a hybrid splitting are constructed for solving our existing EHL model defined in Section~\ref{sec:one}.
In Section~\ref{sec:five}, local Fourier analysis is performed to calculate quantitative estimate of splitting calculated in Section~\ref{sec:three}.
In Section~\ref{sec:six}, numerical experiments are conducted to check the performance of present splitting and its improvement to EHL model.
At the end of Section~\ref{sec:seven}, overall conclusion is summarized.
\section{Preliminaries}\label{section:s2}
In this section, our main goal is to introduce few prerequisite theory which already used in our computation and cannot be
ignored or avoided in the present analysis. Above nonlinear variational inequalities is solved numerically by using fixed
point iteration theory \cite{v1,free_boundary,ALubrecht}. The main challenge appears here in the form of producing a stable 
iterative smoother for EHL inequalities when the applied load on contacting bodies in EHL model become sufficiently large and after
few iterations solution start blowing up.
In such cases, iterative smoother for solving such model is stable only if nonlocal effect produced by film thickness equation 
is controlled by small change calculation in the iteration to make the overall effect local in updated pressure value.
This effect is reduced by introducing special iterative smoother known as {\bf distributive smoother}
\cite{vr,MLMI,Dinar,wittum}. The advantage of adopting such relaxation diminishes aggregation in film thickness computation 
and eventually leads to stable relaxation. Therefore, we need an extra care for computing film thickness term during each iteration.
Let us define deformation integral $\mathscr{D}_{f}$ as 
\begin{align}
\label{eqn14}
  \mathscr{D}_{f}(x,y) = \frac{2}{\pi^2}\int\limits_{-\infty}^{\infty} \int\limits_{-\infty}^{\infty}
 \frac{u(x^{'},y^{'})}{\sqrt{(x-x^{'})^2+(y-y^{'})^2}}dx^{'}dy^{'}.
\end{align}
We approximate the above integral Eqn.~\ref{eqn14} taking pressure $u$ as piecewise constant function namely $u^{hh}_{i',j'}$ on 
sub-domain
\begin{align}
\label{eqn15}
\Omega^{hh}=\Big\{ (x,y) \in \mathbb{R}^{2}\Big|x_{i^{'}}-\frac{h}{2} \le x \le x_{i^{'}}
+\frac{h}{2},y_{j^{'}}-\frac{h}{2} \le y \le y_{j^{'}}
+\frac{h}{2} \Big\}.
\end{align} 
and discrete deformation
\begin{align}
\label{eqn16}
 {\mathscr{D}_{f}}_{i,j} 
 = \mathscr{D}_{f}(x_{i},y_{j})\approx\frac{2}{\pi^2}\sum_{i'=0}^{n_{x}}
 \sum_{j'=0}^{n_{y}}{\mathscr{G}^{hh}}_{i,i^{'},j,j^{'}}u^{hh}_{i',j'},
\end{align}
where the coefficients $\mathscr{G}^{hh}_{i,i^{'},j,j^{'}}$ is written as
\begin{align}
\label{eqn17}
\mathscr{G}^{hh}_{i,i^{'},j,j^{'}} = \int\limits_{x_{i^{'}}-\frac{h}{2}}^{x_{i^{'}}
+\frac{h}{2}} \int\limits_{y_{j^{'}}-\frac{h}{2}}^{y_{j^{'}}+\frac{h}{2}}
\frac{1}{\sqrt{(x-x^{'})^2+(y-y^{'})^2}}dx^{'}dy^{'}
\end{align}
and evaluated analytically.
Above integration Eqn.~\ref{eqn17} yields nine different results for the cases that are defined as
\[ x_{i} < x_{i^{'}},  x_{i} > x_{i^{'}}, x_{i} = x_{i^{'}} \text{ and }
 y_{j} < y_{j^{'}}, y_{j} > y_{j^{'}}, y_{j} = y_{j^{'}} \]
respectively. The nine results are combined into one expression
\small
\begin{align}
\label{eqn18}
\mathscr{G}^{hh}_{i,i^{'},j,j^{'}} =\frac{2}{\pi^{2}}\Big\{
 |x_{+}|\sinh^{-1}(\frac{y_{+}}{x_{+}})+|y_{+}|\sinh^{-1}(\frac{x_{+}}{y_{+}}) 
-|x_{-}|\sinh^{-1}(\frac{y_{+}}{x_{-}}) \nonumber \\ -|y_{+}|\sinh^{-1}(\frac{x_{-}}{y_{+}})   
-|x_{+}|\sinh^{-1}(\frac{y_{-}}{x_{+}})-|y_{-}|\sinh^{-1}(\frac{x_{+}}{y_{-}})\nonumber \\
+|x_{-}|\sinh^{-1}(\frac{y_{-}}{x_{-}})+|y_{-}|\sinh^{-1}(\frac{x_{-}}{y_{-}}) \Big\},
\end{align}
\normalsize
where 
\begin{align*}
 x_{+} = x_{i}-x_{i^{'}}+\frac{h}{2},\quad
 x_{-} = x_{i}-x_{i^{'}}-\frac{h}{2}  \nonumber \\ 
 y_{+} = y_{j}-y_{j^{'}}+\frac{h}{2},\quad 
 y_{-} = y_{j}-y_{j^{'}}-\frac{h}{2}.
\end{align*}
Therefore film thickness in discretized form is written as
\begin{equation}
\label{eqn19}
 \mathscr{H}_{i,j}^{hh} := \mathscr{H}_{00}+\frac{x^2_{i}}{2}+\frac{y^{2}_{j}}{2}
 +\sum_{i'} \sum_{j'}\mathscr{G}^{hh}_{|i-i'|,|j-j'|}{u}_{i^{'},j^{'}}^{hh} ={\text{\tiny H}\mathscr{F}^{h}_{i,j}},
\end{equation}
where $\text{\tiny H}\mathscr{F}^{h}$ is right hand of the film thickness. For computing above discrete film thickness Eqn.~\ref{eqn19},
small change using relaxation is measured as
\begin{align}
\label{eqn20}
 \sigma^{h}_{i,j} = \frac{r_{i,j}^h}{\mathscr{G}^{hh}_{0,0}}, 
\end{align}
where $\mathscr{G}^{hh}_{0,0} = \mathscr{G}^{hh}_{i=i',j=j'}$ and the residual ${r_{\text{J}}^{h}}_{i,j}$ for Jacobi relaxation 
is given by
\begin{align}
\label{eqn21}
 {r_{\text{J}}^{h}}_{i,j}=\text{\tiny H}{\mathscr{F}^{h}_{i,j}}-\mathscr{H}_{00}
 -\frac{x^2_{i}}{2}-\frac{y^{2}_{j}}{2}-\sum_{i'} \sum_{j'}\mathscr{G}^{hh}_{|i-i'|,|j-j'|}\tilde{u}_{i,j}^{h} 
\end{align}
For Gauss-Seidel relaxation, residual ${r_{GS}}^{h}_{i,j}$ is given by
\begin{align}
\label{eqn22}
 {r_{GS}}^{h}_{i,j}=\text{\tiny H}{\mathscr{F}^{h}_{i,j}}-\mathscr{H}_{00}-\frac{x^2_{i}}{2}-\frac{y^{2}_{j}}{2}\nonumber    \\
                  -\sum_{i'<i} \sum_{j'}\mathscr{G}^{hh}_{|i-i'|,|j-j'|}\bar{{u}}_{i,j}^{h}
                  -\sum_{i'=i} \sum_{j'<j}\mathscr{G}^{hh}_{|i-i'|,|j-j'|}\tilde{{u}}_{i,j}^{h}   \nonumber    \\
                  -\sum_{i'=i} \sum_{j'>=j}\mathscr{G}^{hh}_{|i-i'|,|j-j'|}\tilde{{u}}_{i,j}^{h}
                  -\sum_{i'>i} \sum_{j'}\mathscr{G}^{hh}_{|i-i'|,|j-j'|}\tilde{{u}}_{i,j}^{h},
\end{align}
where $\tilde{u}_{i,j}$ and $\bar{u}_{i,j}$ old and new updated values of pressure respectively.
\subsubsection{Smooth kernel computation using MLMI}
Suppose we want to solve integral of type Eqn.~\ref{eqn19}. If kernel $\mathscr{G}(x,y)$ is sufficiently smooth with respect to 
the variable $y$, we approximate discrete kernel $\mathscr{G}^{hh}_{i,j}$ by high order interpolation operator as
\vspace{-0.8em}
\begin{gather}
\label{eqn23}
 \tilde{\mathscr{G}}^{hh}_{i,j} \simeq [\mathcal{I}^{h}_{H}\mathscr{G}^{hH}_{i,.}]_{j},
\end{gather}
where the high order interpolation operator is denoted by $\mathcal{I}^{h}_{H}$ and $\mathscr{G}^{hH}_{i,.}$
is \emph{injected} from $\mathscr{G}^{hh}_{i,.}$ \emph{i.e.}, $\mathscr{G}^{hH}_{i,J}\stackrel{\text{def}}{=} \mathscr{G}^{hh}_{i,2J}$.
Superscript $h$ and $H$ denote the finer and the coarser grid respectively.
Then the finer grid integral computation of Eqn.~\ref{eqn19} is approximated on coarser grid in following way 
\begin{align}
\label{eqn24}
 \mathcal{W}^{h}_{i}\simeq \tilde{\mathcal{W}}^{h}_{i} \stackrel{\text{def}}
 {=} h^{d}\sum_{j}\tilde{\mathscr{G}}^{hh}_{i,j}{u^{*}}^{h}_{j}
 =h^{d}\sum_{j}[\mathcal{I}^{h}_{H}\mathscr{G}^{hH}_{i,.}]_{j}{u^{*}}^{h}_{j} \nonumber \\
 = h^{d}\sum_{J}\mathscr{G}^{hH}_{i,J}[(\mathcal{I}^{h}_H)^{T}{u^{*}}^{h}_{.}]_{J} 
 = H^{d}\sum_{J}\mathscr{G}^{hH}_{i,J}{u^{*}}^{H}_{J},
\end{align}
where 
\begin{align}
\label{eqn25}
{u^{*}}^{H}_{J}\stackrel{\text{def}}{=}2^{-d}[(\mathcal{I}^{h}_{H})^T{u^{*}}^{h}_{.}]_J.
\end{align}
Whenever kernel $\mathscr{G}(x,y)$ is also smooth enough with respect to $x$ variable, the discrete sum $\mathcal{W}^{h}_{i}$
is evaluated on coarse grid points $i=2I$ by use of high order interpolation operator $\mathcal{\hat{I}}^{h}_{H}$. It is written as 
\begin{align}
\label{eqn26}
 \mathcal{W}^{h} \simeq \mathcal{\hat{I}}^{h}_{H} \mathcal{W}^H,
\end{align}
where 
\begin{align}
\label{eqn27}
 \mathcal{W}^{H}_{I} \stackrel{\text{def}}{=} {\tilde{\mathcal{W}}}^{h}_{2I} = H^{d}\sum_{J}\mathscr{G}^{HH}_{I,J}{u^{*}}^{H}_{J}
\end{align}
and where $\mathscr{G}^{HH}_{.,J}$ is \emph{injected} from $\mathscr{G}^{hH}_{.,J}$, \emph{i.e.},
$\mathscr{G}^{HH}_{I,J}\stackrel{\text{def}}{=}
\mathscr{G}^{hH}_{2I,J} = \mathscr{G}^{hh}_{2I,2J} $.
\subsubsection{Singular-Smooth or mild singular Kernel computation using MLMI}
In general, kernel $\mathscr{G}$ has a mild singularity near a point $x=y$. We rewrite our coarse grid approximation
by adding correction term near singularity in the following way (see \cite{MLMI})
\begin{align}
\label{eqn28}
 \mathcal{W}^{h}_{i} = h^{d}\sum_{j}\mathscr{G}^{hh}_{i,j}{u^{*}}^{h}_{j}
 =h^{d}\sum_{j}\tilde{\mathscr{G}}^{hh}_{i,j}{u^{*}}^{h}_{j}+
 h^{d}\sum_{j}(\mathscr{G}^{hh}_{i,j}-\tilde{\mathscr{G}}^{hh}_{i,j}){u^{*}}^{h}_{j} \nonumber \\
 = h^{d}\sum_{j}[\mathcal{I}^{h}_{H}\mathscr{G}^{hH}_{i,.}]_{j}{u^{*}}^{h}_{j}+h^{d}\sum_{j}(\mathscr{G}^{hh}_{i,j}-
 \tilde{\mathscr{G}}^{hh}_{i,j}){u^{*}}^{h}_{j} \nonumber \\
 = \mathcal{W}^{H}_{I} +h^{d}\sum_{j}(\mathscr{G}^{hh}_{i,j}-\tilde{\mathscr{G}}^{hh}_{i,j}){u^{*}}^{h}_{j}
\end{align}
Since $\tilde{\mathscr{G}}^{hh}_{i,j}$ is an interpolation of $\mathscr{G}^{hh}_{i,j}$ itself using coarse grid points, the
operator $(\mathscr{G}^{hh}_{i,j}-\tilde{\mathscr{G}}^{hh}_{i,j})$ is given by
\begin{align}
\label{eqn29}
(\mathscr{G}^{hh}_{i,j}-\tilde{\mathscr{G}}^{hh}_{i,j}) =
  \begin{cases}
    0       & \quad  \quad         j=2J \\
    O(h^{2p}\mathscr{G}^{2p}(\xi)  & \quad \text{otherwise }, \\
  \end{cases}
\end{align}
where $2p$ is the interpolation order and $\mathscr{G}^{2p}(\xi)$ is a $2p^{th}$ derivative of $\mathscr{G}$ at some 
intermediate point $\xi$. Thus if the derivative of $\mathscr{G}$ becomes small, the correction term 
become small and can be neglected. However, in case of singular smooth kernel ($i \simeq j$),
we require the corrections in a neighborhood of $i= j(||j-i|| \le m \quad \text{or} \quad i-m \le j \le i+m)$.
Thus Eq.~(\ref{eqn28}) is simplified as follows
\begin{align}
\label{eqn30}
 \mathcal{W}^{h}_{i} =\mathcal{W}^{H}_{I} 
 +h^{d}\sum_{||j-i|| \le m}(\mathscr{G}^{hh}_{i,j}-\tilde{\mathscr{G}}^{hh}_{i,j}){u^{*}}^{h}_{j}
\end{align}
Advantage of using multi-level procedure in  film thickness $\mathscr{H}$ computation reduces integral complexity up to $O(n\log n)$. 
A schematic diagram of multi level multi integration procedure is given in Fig.~\ref{fig:mlm}.
       \begin{figure}
        \centering
        \includegraphics[width=10cm,height=12cm,keepaspectratio]{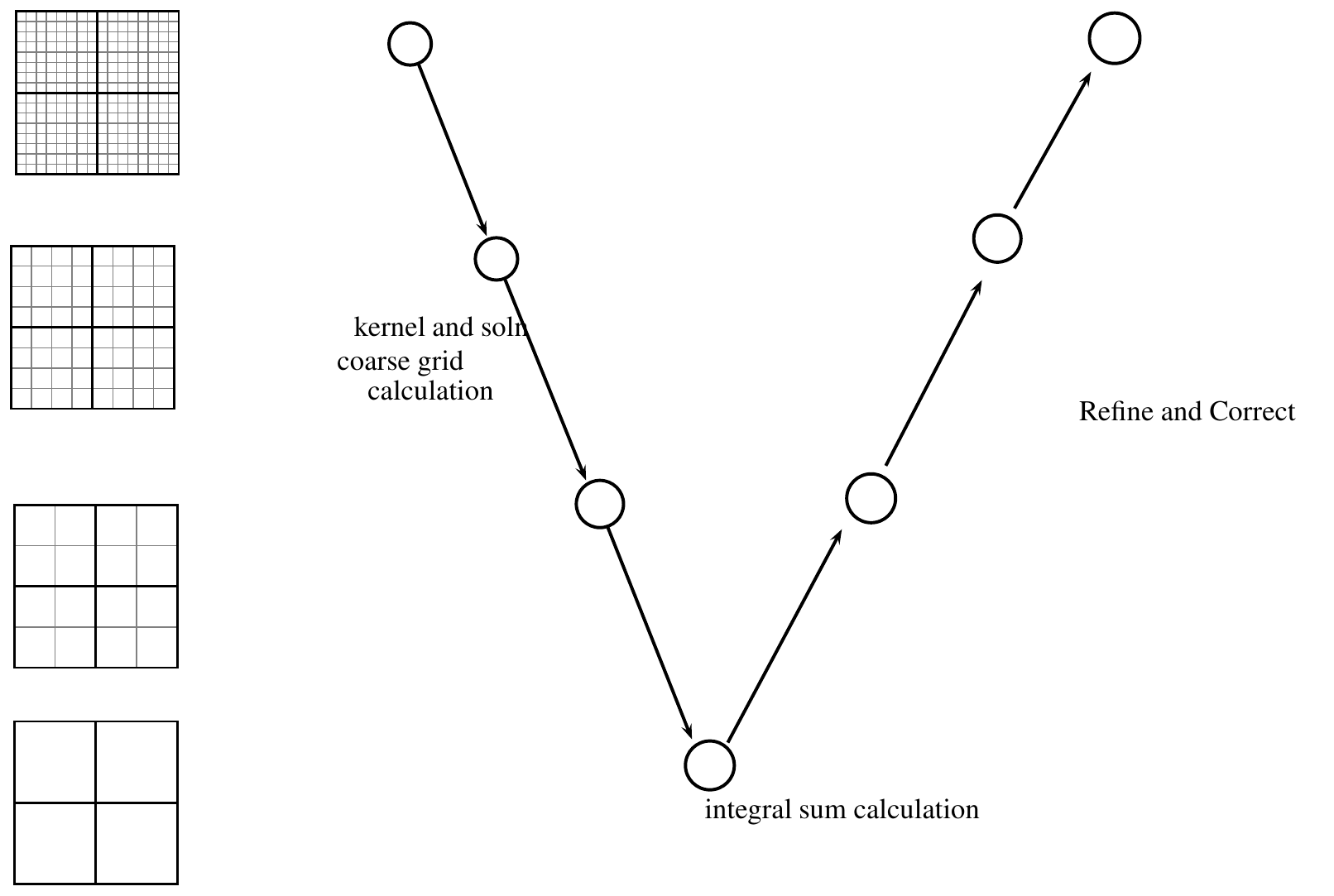}
        \caption{Schematic diagram of multi level multi integration}
        \label{fig:mlm}
         \end{figure}
\subsection{Multi-Grid Method for variational inequality arising in EHL Problem}
In this section, we discuss multi-grid method \cite{hackbus1985,mg_adopt} for variational inequality of EHL model.
EHL problem is viewed as a linear complementarity problem \cite{free_boundary,osterlee2003} of the form
\begin{align}
\label{eqn31}
 L u \le f_{1} \quad x \in \Omega  \nonumber \\
 u\ge {f}_{2} \quad x \in \Omega  \nonumber \\
 u = g \quad x \in \partial \Omega \nonumber \\
 (u-f_{2})(Lu - f_{1}) = 0 \quad x \in \Omega,
\end{align}
where $L$ is a linear differential operator.
We want to solve the problem in discrete hierarchical sub-domains of the following form 
\begin{align}
\label{eqn32}
\Big\{ \Omega_{l};\quad \Omega_{l-1} \subset \Omega_{l} \subset \Omega \quad \forall l \in \mathbb{Z}
\cap [1, M],\text{ where $M$}\in \mathbb{R} \Big \}
\end{align}
Hence discrete form of complementarity problem on level $l$ is written as 
\begin{align}
 L_{l}u_{l} \le f_{1,l} \quad x_{l} \in \Omega_{l}  \nonumber \\
 u_{l}\ge f_{2,l} \quad x_{l} \in \Omega_{l} \nonumber \\
 u_{l}= g \quad x_{l} \in \partial \Omega_{l} \nonumber \\ \label{eqn33}
 (u_{l}-f_{2,l})(L_{l}u_{l} - f_{1,l}) = 0 \quad x_{l} \in \Omega_{l}. 
\end{align}
Let $u_{l}$ and $v_{l}$ are an exact solution and approximated solution of above LCP Eqn.~\ref{eqn33}.
Suppose that the error $e_{l}={u}_{l}-{v}_{l} $ is smooth after the iteration sweeping.
Then complementarity problem satisfied for error equation $e_{l}$ on finer level is read as
\begin{align}
\label{eqn34}
 L_{l}{e}_{l}\le {r}_{l} \quad {\bf x} \in \Omega \nonumber \\
 {e}_{l}+v_{l} \ge f_{2,l} \quad {\bf x} \in \Omega \nonumber \\
 (e_{l}+v_{l}-f_{2,l})(L_{l}e_{l}-r_{l})=0 \quad {\bf x} \in \Omega,
\end{align}
where residual $r_{l}=f_{1,l}-L_{l}v_{l}$. 
Such smooth error $e_{l}$ is approximated on a coarse grid without loosing any essential information.
The LCP coarse grid equation for the coarse grid approximation of the error
${e}_{l-1}$ is therefore defined in PFAS by
\begin{align}
 L_{l-1}e_{l-1}\le I_{l}^{l-1}{r}_{l} \nonumber \\
 {e}_{l-1}+\tilde{I}_{l}^{l-1}v_{h} \ge f_{2,l-1} \nonumber \\ \label{eqn35}
 (e_{l-1}+\tilde{I}_{l}^{l-1}v_{h}-f_{2,l-1})(L_{l-1}{e}_{l-1}-I_{l}^{l-1}r_{l})=0. 
\end{align}
Since the problem is nonlinear and we are solving inequalities, we solve for full approximation
${v}_{l-1}={e}_{l-1}+I_{l}^{l-1}{v}_{l}$ but interpolate only $v_{l-1}$ back to fine grid.
The main difference between multi-grid methods for equations and inequalities occur due to fact that,
in case of fine grid converged solution $v_{l}= v_{l}^{*}$ the coarse grid correction equation should be zero. 
Consequently, we have the following relation
\begin{align}
\label{eqn36}
 I_{l-1}^{l}{e}_{l-1}=I^{l}_{l-1}({{v}^{*}}_{l-1}-\tilde{I}^{l-1}_{l}{{v}^{*}}_{l})=0 \Rightarrow v_{l-1}=\tilde{I}^{l-1}_{l}v_{l}
\end{align}
(assume that operator $I^{l}_{l-1}$ keeps nonzero quantities nonzero).\\
Furthermore, for a converged solution of fine grid LCP problem the coarse grid correction
provides us the following condition on restriction operators,
\begin{align}
 I_{l}^{l-1}(f_{1,l}-L_{l}{v}_{l}) \ge 0 \nonumber \\
 \tilde{I}^{l-1}_{l}v_{h} \ge f_{2,l-1} \nonumber \\ \label{eqn37}
 (\tilde{I}^{l-1}_{l}v_{l}-f_{2,l-1})^{T}I_{l}^{l-1}(f_{1,l}-L_{l}v_{l})=0 
\end{align}
Since $f_{1,l}-L_{l}v_{l}\equiv 0$ for any converge solution. 
Hence above inequalities~\ref{eqn37} will satisfy for any rational choice of restriction operators $I^{l-1}_{l}$ and $\tilde{I}^{l-1}_{l}$.
For capturing free boundary and for achieving fast convergence the bilinear interpolation operator $I_{l-1}^{l}$ is implemented only
for unknowns on the {\bf inactive} points that means, 
\begin{align}
\label{eqn38}
 v_{l} \Leftarrow v_{l}+ I_{l-1}^{l}{e}_{l-1} \quad \text{if} \quad v_{l} > f_{2,l} \nonumber \\
 v_{l} \Leftarrow v_{l}\quad \text{elsewhere} \quad (v_{l} = f_{2,l}).
\end{align}
\section{Linear study for convection-diffusion problem}\label{sec:three}
Our specific interest in this Section is to develop an robust splitting for our EHL model. Such splitting is constructed by imitating
series of linear model problem one by one. First we consider well known convection-diffusion problem of the form
\begin{example}\label{ex:one}
 \begin{align}
 \label{eqn39}
   L u = (a(x,y) u)_{x}-\epsilon \Delta u = f(x,y) \quad \forall x,y \in \Omega \nonumber \\
  u(x,y) = g(x,y) \quad \forall x,y \in \partial \Omega,
  \end{align}
 \end{example}
where $0 < \epsilon < < 1$ (note that we do not have any $y$ derivative in convection term).
Then discretization of convective term for $(a u)_{x}$ is performed as
 \begin{align}
 \label{eqn40}
  (a u)_{x}=\frac{a}{h}(u_{i,j}-u_{i-1,j})=: L_{1} 
 \end{align}
However, this scheme is only $O(h)$ accurate. Our interest here to increase accuracy at least smooth part without contaminating
any wiggle in solution. Consider the Van Leer's $\kappa$-schemes \cite{vanleer} for discretization
term  $(a u)_{x}$ (for $a = \text{const} > 0$) as 
\begin{gather}
\label{eqn41}
 (a u)_{x}=\frac{a}{h}[(u_{i,j}-u_{i-1,j})-\frac{\kappa}{2}(u_{i,j}-u_{i-1,j})+\frac{1-\kappa}{4}(u_{i,j}-u_{i-1,j}) \nonumber\\
 +\frac{1+\kappa}{4}(u_{i+1,j}-
 u_{i,j})-\frac{1-\kappa}{4}(u_{i,j}-u_{i-2,j})]\nonumber\\
 =L_{1} + L_{\alpha}+L_{\beta}+L_{\gamma}+L_{\delta}
\end{gather}
(similar scheme can be constructed for $a < 0$).
The resulting discrete model Example.~\ref{ex:one} by $\kappa$-scheme (take $\kappa = 0$ here) is denoted by
\begin{align}
\label{eqn42}
[L_{\kappa=0}]=\frac{a}{h}\left[\begin{matrix}
                      1/4 \ &  -5/4  &  3/4  & 1/4 & 0 \\
                     \end{matrix}
\right]+\frac{\epsilon}{h^2}
\left[\begin{matrix}
 0 \ & \ -1 \ &\  0 \\
 -1 \ & \ 4 \ &\ -1 \\
   0 \ &\ -1 \ &\ 0 \\
    \end{matrix} 
\right]
\end{align}
In general, above discrete equation.~\ref{eqn39} do not produces $M$-matrix and many iterative splitting on $L_{\kappa}$ diverge.
Therefore, this problem is solved using TVD scheme with help of appropriate flux limiters to prevent a solution from unwanted oscillation.
Now consider $\kappa=-1$ then the second-order upwind scheme looks like ($a > 0$)
\begin{gather}
 (au)_{x}=
 \frac{a}{h}[(u_{i,j}-u_{i-1,j})+\frac{1}{2}(u_{i,j}-u_{i-1,j}) \nonumber\\
 +\frac{1}{2}(u_{i,j}-u_{i-1,j})-\frac{1}{2}(u_{i-1,j}-u_{i-2,j})]\nonumber\\ \label{eqnn43}
 =L_{1}+L_{\alpha}+L_{\gamma}+L_{\delta}.
\end{gather}
We enforce Eqn.~\ref{eqnn43} to satisfy TVD condition by multiply limiter functions in 
the additional terms $L_{\alpha}, L_{\gamma} $ and $L_{\delta}$.
Then following two type of discretization for convection term are presented here as
\begin{gather}
 (au)_{x}=\frac{a}{h}[(u_{i,j}-u_{i-1,j})+\frac{1}{2}\phi(r_{i-1/2})(u_{i,j}-u_{i-1,j})\nonumber \\ \label{eqn44}
 -\frac{1}{2}\phi(r_{i-3/2})(u_{i-1,j}-u_{i-2,j})]
 =L_{1}+L_{\alpha}+L_{\gamma}
\end{gather}
and
\begin{gather}
 (au)_{x}=\frac{a}{h}[(u_{i,j}-u_{i-1,j})+\frac{1}{2}\phi(r_{i-1/2})(u_{i,j}-u_{i-1,j})\nonumber \\
 +\frac{1}{2}\phi(r_{i-3/2})(u_{i,j}-u_{i-1,j})
 -\frac{1}{2}\phi(r_{i-3/2})(u_{i-1,j}-u_{i-2,j})]\nonumber \\\label{eqn45}
 =L_{1}+L_{\alpha}+L_{\beta}++L_{\gamma},
\end{gather}
where $r_{i-1/2}=\dfrac{(u_{i+1,j}-u_{i,j})}{(u_{i,j}-u_{i-1,j})}$ and 
$r_{i-3/2}=\dfrac{(u_{i,j}-u_{i-1,j})}{(u_{i-1,j}-u_{i-2,j})}$.\\
In Fig.~\ref{fig:limtr} represents graph of limiter function $(r,\phi(r))$ on which the resulting convection discretization term
defined in Eqn.~\ref{eqnn43} and Eqn.~\ref{eqn44} enforce to be TVD and higher order accurate (see \cite{Oosterlee}).
       \begin{figure}
        \centering
        \includegraphics[width=10cm,height=12cm,angle =-90,keepaspectratio]{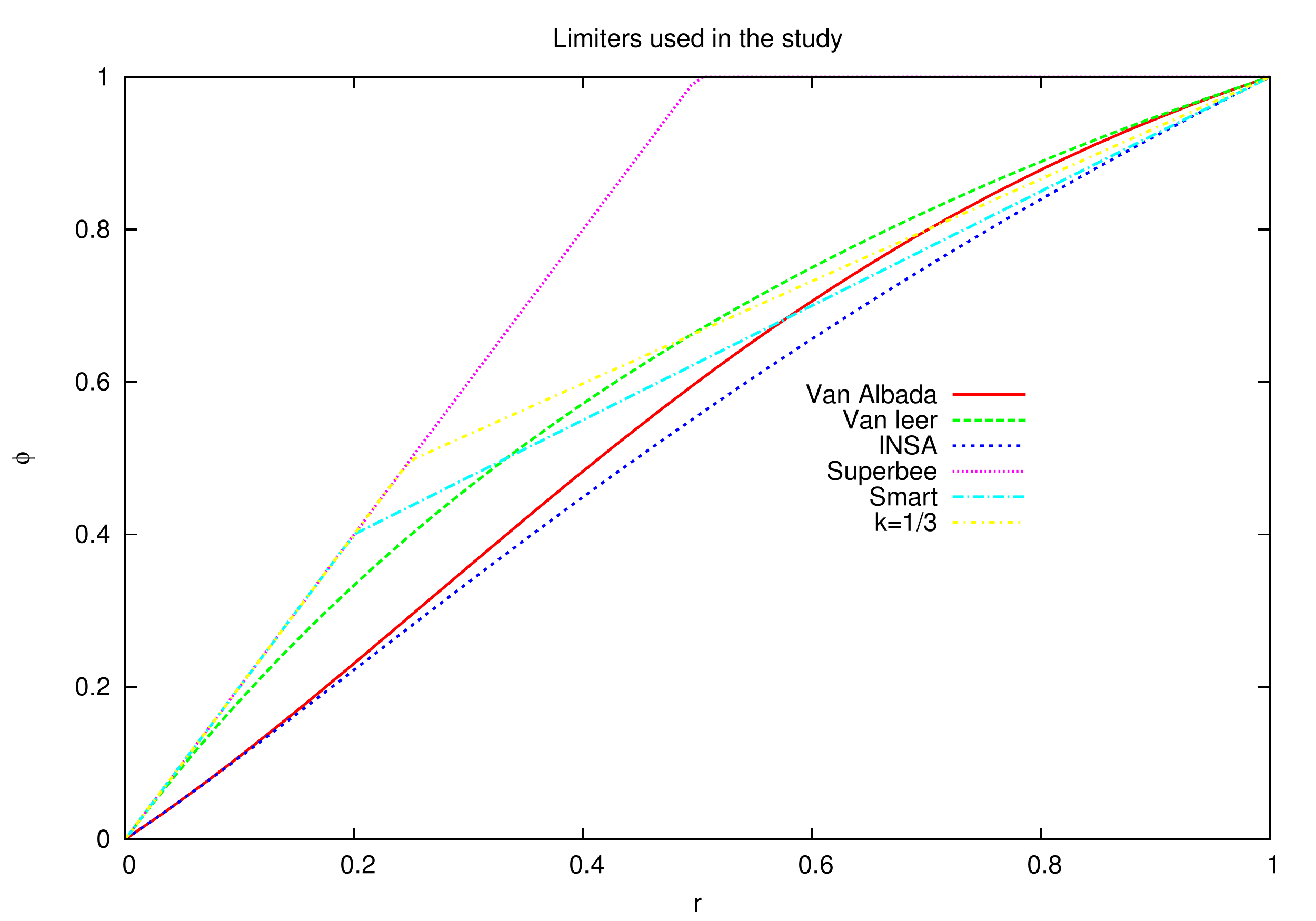}
        \caption{Schematic diagram of class of limiter function $\phi$ with respect to $r$ used in our study (see \cite{Oosterlee})}
        \label{fig:limtr}
         \end{figure}
The discrete representation of Example~\ref{ex:one} using Van-leer $\kappa$-scheme is defined as 
\begin{align}
\label{eqn46}
 L_{\kappa}u = \sum_{l_{x} \in \mathcal{I}}\sum_{l_{y}
 \in \mathcal{I}}\mathcal{C}^{(\kappa)}_{l_{x}l_{y}}u_{i+l_{x},j+l_{y}}.
\end{align}
Moreover, in stencil notation it is represented as
\begin{align}
\label{eqn47}
L_{\kappa} = 
 \begin{pmatrix}
        &         & \mathcal{C}_{02}^{\kappa} &        &        \\
        &         & \mathcal{C}_{01}^{\kappa} &         &         \\
\mathcal{C}_{-20}^{\kappa} & \mathcal{C}_{-10}^{\kappa} & \mathcal{C}_{00}^{\kappa} & \mathcal{C}_{10}^{\kappa} & \mathcal{C}_{20}^{\kappa} \\
        &         & \mathcal{C}_{0-1}^{\kappa} &        &         \\
        &         & \mathcal{C}_{0-2}^{\kappa} &        &              
 \end{pmatrix}
.
\end{align} 
Then the discrete matrix equation $L_{\kappa}u =f$ is solved efficiently by the use of multi-grid. The related splitting is constructed by 
taking the matrix operator defined in Eqn.~\ref{eqn47}. In particular case, the splitting in $x$-direction is scanned as forward 
(or backward direction depending on flow direction) lexicographical order and it is represented as
$S_{\kappa}=S_{\kappa}^{x_{f}}$ (or $S_{\kappa}^{x_{b}}$). For matrix operator $L_{\kappa}$, the forward splitting
$S_{\kappa}^{x_{f}}$ is defined as 
\begin{gather*}
 L_{\kappa}=L^{x}_{\kappa/2}-(L^{x}_{\kappa/2}-L_{\kappa})=:L^{+}_{\kappa}+L^{0}_{\kappa}+L^{-}_{\kappa},
\end{gather*}
where
\begin{gather*}
L^{x}_{\kappa/2}:=L^{+}_{\kappa}+L^{0}_{\kappa}=
 \begin{pmatrix}
        &         & 0 &        &        \\
        &         & 0 &         &         \\
      0 &       0 & 0 & 0       & 0        \\
        &         & \mathcal{C}_{0-1}^{\kappa} &        &         \\
        &         & \mathcal{C}_{0-2}^{\kappa} &        &              
 \end{pmatrix}
 +
 \begin{pmatrix}
        &         & 0 &        &        \\
        &         & 0 &         &         \\
0       & \mathcal{C}_{-10}^{\kappa/2} & \mathcal{C}_{00}^{\kappa/2} & \mathcal{C}_{10}^{\kappa/2} & 0 \\
        &         & 0 &        &         \\
        &         & 0 &        &              
 \end{pmatrix} 
\end{gather*}
and therefore overall splitting is
\begin{gather*}
L^{x}_{\kappa/2}u^{n+1}= (L^{x}_{\kappa/2}-L_{\kappa})u^{n}+f.
\end{gather*}
Now for a fixed $x$-line ($m$-grid points in $x$-direction) $(i,j_{0})_{( 1\le i \le m)}$, we have the following
\begin{gather*}
 L^{0}_{\kappa}u^{*}=f+L^{0}_{\kappa}u^{n}-(L^{-}_{\kappa}+L^{0}_{\kappa})u^{n}-L^{+}_{\kappa}u^{n+1}.
\end{gather*}
$L^{0}_{\kappa}$ corresponds the operator to the unknowns $u^{*}$ which are scanned simultaneously. $L^{-}_{\kappa}$ corresponds
the operator to the old approximation $u^{n}$, and $L^{+}_{\kappa}$ operator having updated values of $u^{n+1}$.
Now by applying under-relaxation constant $\omega$ in above equation we have
\begin{align*}
u^{n+1}=u^{*}\omega+u^{n}(1-\omega),
\end{align*}
therfore splitting equation can be rewritten in corresponding change, $\sigma^{n+1}=u^{n+1}-u^{n}$ form as
\begin{align*}
 L^{0}_{\kappa}\sigma^{n+1}=f-(L^{-}_{\kappa}+L^{0}_{\kappa})u^{n}-L^{+}_{\kappa}u^{n+1},\\
 u^{n+1}=u^{n}+\sigma^{n+1}\omega
\end{align*}
Now we construct series of splitting for solving Eqn.~\ref{eqn39} as below.\\
\underline{\bf Splitting : $L_{s0}$}
This splitting is constructed by taking upwind operator $L_{1}$ plus a ``positive" part of the second-order 
operators $L_{\alpha}$ and $L_{\beta}$ from Eqn.~\ref{eqn45} and part of diffusion operator from Eqn.~\ref{eqn47}.
\begin{gather}
\label{eqn48}
 L_{\kappa}^{0}u=-\Big\{\frac{\epsilon}{h^2}+\frac{a}{4h}(5-3\kappa)\Big\}u_{i-1,j}+\Big\{\frac{a}{h}\Big(\frac{2-\kappa}{2}
+\frac{1-\kappa}{4}\Big)+\frac{4\epsilon}{h^2}\Big\}u_{i,j}\nonumber \\
+\Big\{-\frac{\epsilon}{h^{2}}\Big\}u_{i+1,j}\nonumber\\
 L_{\kappa}^{+}u=\Big\{-\frac{\epsilon}{h^{2}}\Big\}u_{i,j-1}\nonumber\\
 L_{\kappa}^{-}u=\Big\{\frac{a}{h}\Big(\frac{1-\kappa}{4}\Big)\Big\}u_{i-2,j}+\Big\{\frac{a}{h}\Big(\frac{1-\kappa}{4}\Big)
 \Big\}u_{i-1,j}+\Big\{-\frac{a}{h}\Big(\frac{1+\kappa}{4}\Big)\Big\}u_{i,j}\nonumber \\
 +\Big\{\frac{a}{h}\Big(\frac{1+\kappa}{4}\Big)\Big\}u_{i+1,j}+\Big\{-\frac{\epsilon}{h^{2}}\Big\}u_{i,j+1}.
\end{gather}
\underline{\bf Splitting : $Ls1$}
This splitting is constructed taking upwind operator $L_{1}$ plus a ``positive" part of the second-order 
operators $L_{\alpha}$ from Eqn.~\ref{eqn44} and part of diffusion operator from Eqn.~\ref{eqn47}.
\begin{gather}
\label{eqn49}
 L_{\kappa}^{0}u=\Big\{-\frac{a}{h}\Big(\frac{2-\kappa}{2}\Big)-\frac{\epsilon}{h^2}\Big\}u_{i-1,j}+\Big\{\frac{a}{h}\Big(\frac{2-\kappa}{2}\Big)
 +\frac{4\epsilon}{h^2}\Big\}u_{i,j}+\Big\{-\frac{\epsilon}{h^{2}}\Big\}u_{i+1,j}\nonumber \\
 L_{\kappa}^{+}u=\Big\{-\frac{\epsilon}{h^{2}}\Big\}u_{i,j-1}\nonumber\\
 L_{\kappa}^{-}u=\Big\{\frac{a}{h}\Big(\frac{1-\kappa}{4}\Big)\Big\}u_{i-2,j}+\Big\{\frac{a}{h}\Big(\frac{1-\kappa}{4}\Big)
 \Big\}u_{i-1,j}+\Big\{-\frac{a}{h}\Big(\frac{1+\kappa}{4}\Big)\Big\}u_{i,j}\nonumber\\
 +\Big\{\frac{a}{h}\Big(\frac{1+\kappa}{4}\Big)\Big\}u_{i+1,j}+\Big\{-\frac{\epsilon}{h^{2}}\Big\}u_{i,j+1}
\end{gather}
\underline{\bf Splitting : ${Ls2}$}
In this case splitting coefficients $\mathcal{C}_{**}^{\kappa}$ correspond only to the first-order upwind operator $L_{1}$ of a 
discretized Eqn.~\ref{eqn44} plus diffusion operator.
\begin{gather}
\label{eqn50}
 L_{\kappa}^{0}u=\Big\{-\frac{a}{h}-\frac{\epsilon}{h^{2}}\Big\}u_{i-1,j}+\Big\{\frac{a}{h}
 +\frac{4\epsilon}{h^2}\Big\}u_{i,j}+\Big\{-\frac{\epsilon}{h^{2}}\Big\}u_{i+1,j} \nonumber \\
 L_{\kappa}^{+}u=\Big\{-\frac{\epsilon}{h^{2}}\Big\}u_{i,j-1}\nonumber\\
 L_{\kappa}^{-}u=\Big\{\frac{a}{h}\Big(\frac{1-\kappa}{4}\Big)\Big\}u_{i-2,j}+\Big\{-\frac{a}{h}\Big(\frac{1-3\kappa}{4}\Big)
 \Big\}u_{i-1,j}+\Big\{-\frac{a}{h}\Big(\frac{1+3\kappa}{4}\Big)\Big\}u_{i,j} \nonumber\\
 +\Big\{\frac{a}{h}\Big(\frac{1+\kappa}{4}\Big)\Big\}u_{i+1,j}+\Big\{-\frac{\epsilon}{h^{2}}\Big\}u_{i,j+1}
\end{gather}
\underline{\bf Splitting : $Ls3$}
The third splitting named as $\kappa$- distributive line relaxation is constructed by assuming a ghost variable $\sigma_{*}$ (with the 
same cardinality as $\sigma$) such that $\sigma = \mathcal{D}\sigma_{*}$, where matrix $\mathcal{D}$ comes due to distributive change of the relaxation.i.e.
We construct line-wise distributive splitting as
\begin{gather}
\label{eqn51}
 u_{i,j}^{n+1}=u^{n}_{i,j}+\sigma_{i,j}-\frac{(\sigma_{i+1,j}+\sigma_{i-1,j}+\sigma_{i,j+1}+\sigma_{i,j-1})}{4}
\end{gather}
This splitting is understood in the following way:
First, discretize Example~\ref{ex:one} by $\kappa$-scheme and get the equation of the form as
\begin{gather*}
L^{x}_{\kappa/2}u^{n+1}= f', \quad \text{where }  f'=(L^{x}_{\kappa/2}-L_{\kappa})u^{n}+f.
\end{gather*}
Now in the above splitting equation put the value of $u^{n+1}$ from Eqn.~\ref{eqn51} and apply distributive splitting in the form of right preconditioner defined below.
\begin{gather*}
L^{x}_{\kappa/2}\sigma^{n+1}= R^{n}\quad \text{and } L^{x}_{\kappa/2}\mathcal{D}\sigma^{n+1}_{*}= R^{n},
\end{gather*}
where the updated change in pressure and residual equation are denoted as
\begin{align*}
 \sigma^{n+1}=\mathcal{D}\sigma^{n+1}_{*} \text{ and } R^{n}=L^{x}_{\kappa/2}u^{n+1}- f'
\end{align*}
respectively.
In other way, line distributive splitting consists of following two steps;
In first step it calculates new ghost value approximation change $\sigma^{n+1}_{*}$. Second step calculates new approximation
change $\sigma^{n+1}$.\\
Now applying above splitting along the $x$-direction in Example~\ref{ex:one}, the diffusive term is computed as
\begin{gather}
\label{eqn52}
 -\epsilon \Big[\Big\{ u_{i+1,j}+\sigma_{i+1}-\frac{(\sigma_{i} + \sigma_{i+2})}{4} \Big\}
 -\Big\{ u_{i,j}+\sigma_{i}-\frac{(\sigma_{i-1} + \sigma_{i+1})}{4} \Big\}\Big]\Big/h^{2}\nonumber\\
 -\epsilon \Big[\Big\{ u_{i-1,j}+\sigma_{i-1}-\frac{(\sigma_{i-2} + \sigma_{i})}{4} \Big\}
 -\Big\{ u_{i,j}+\sigma_{i}-\frac{(\sigma_{i-1} + \sigma_{i+1})}{4} \Big\}\Big]\Big/h^{2}\nonumber\\
 -\epsilon \Big[\Big\{ u_{i,j+1}-\frac{\sigma_{i}}{4} \Big\}
 -\Big\{ u_{i,j}+\sigma_{i}-\frac{(\sigma_{i-1} + \sigma_{i+1})}{4} \Big\}\Big]\Big/h^{2}\nonumber\\
 -\epsilon \Big[\Big\{ u_{i,j-1}-\frac{\sigma_{i}}{4} \Big\}
 -\Big\{ u_{i,j}+\sigma_{i}-\frac{(\sigma_{i-1} + \sigma_{i+1})}{4} \Big\}\Big]\Big/h^{2}.
\end{gather}
and convection term is computed as
\begin{gather}
\label{eqn53}
+\Big[\frac{a_{i+1/2,j}(2+\kappa)}{2h}\Big\{ u_{i,j}+\sigma_{i}-\frac{(\sigma_{i-1} + \sigma_{i+1})}{4} \Big\}\nonumber \\
 -\frac{a_{i-1/2,j}(2+\kappa)}{2h}\Big\{ u_{i-1,j}+\sigma_{i-1}-\frac{(\sigma_{i-2} + \sigma_{i})}{4} \Big\}\Big]
\end{gather}
Other part of convective term which comes from van-leer discretization do not contain any distributive term 
as above explained and kept in right hand side during relaxation and overall splitting is written as follows
\begin{gather}
\Big(\frac{\epsilon}{4h^{2}}+\frac{a_{i-1/2,j}(2+\kappa)}{8h}\Big)\sigma_{i-2}
-\Big(\frac{7\epsilon}{4h^{2}}+\frac{a_{i+1/2,j}(2+\kappa)}{2h}+
\frac{a_{i-1/2,j}(2+\kappa)}{8h}\Big)\sigma_{i-1} \nonumber \\
+\Big(\frac{20\epsilon}{4h^{2}}+\frac{a_{i+1/2,j}(2+\kappa)}{2h}
+\frac{a_{i-1/2,j}(2+\kappa)}{8h}\Big)\sigma_{i}\nonumber \\ 
-\Big(\frac{8\epsilon}{4h^{2}}+\frac{a_{i+1/2,j}(2+\kappa)}{2h}\Big)\sigma_{i+1}
+\frac{\epsilon}{4h^{2}}\sigma_{i+2} \nonumber \\ \label{eqn54}
=R_{i,j}+\Big\{\frac{1+\kappa}{4}(u_{i+1,j}-u_{i,j})-\frac{1-\kappa}{4}(u_{i-1,j}-u_{i-2,j})\Big\}\Big]
\end{gather}
after solving above equation for $\sigma$ along $x$ line direction updated solution $u^{n+1}$ is evaluated as 
\begin{gather*}
 u_{i,j}^{n+1}=u^{n}_{i,j}+\sigma_{i,j}-\frac{(\sigma_{i+1,j}+\sigma_{i-1,j}+\sigma_{i,j+1}+\sigma_{i,j-1})}{4}.
\end{gather*}
However, above splitting $Ls3$ Eqn.~\ref{eqn54} is not robust and very rarely use in practice.\\
We are now interested in showing convergence of LCP through the above presented splitting.
Let us consider domain $\Omega \in \mathbb{R}^2$ with boundary $\partial \Omega$, and consider known functions $f$ and $g$. 
Then find $u$ in a weak sense such that these inequalities hold
\begin{example}
 \begin{align*}
  -(a(x,y)h(u))_{x}+\epsilon \Delta u \le f(x,y) \quad \forall x,y \in \Omega \nonumber\\
   u(x,y) \ge 0 \quad \forall x,y \in \Omega,\\
   u(x,y)[(a(x,y)h(u))_{x}-\epsilon \Delta u - f(x,y)]=0  \quad \forall x,y \in \Omega,\\
  u(x,y) = g(x,y) \quad \forall x,y \in \partial \Omega.
 \end{align*}
\end{example} 
Therefore, discrete version of above problem (finite difference or finite volume) is written in the matrix form
 \begin{align}
 \label{eqnnew1}
  Lu \le f, \nonumber \\
   u \ge 0,\nonumber \\
   u[L u - f]=0,
  \end{align}
where $L$ is a $M$-matrix of order $m\times m$, $u$ and $f$ are $m\times 1$-column vector.   
It is well known that solving above discrete problem is equivalent to solving quadratic
minimization problem of the form
\begin{align} 
\label{eqnnew2}
 G(u)=\frac{1}{2}u^{T}Lu-f^{T}u, \nonumber \\
 \min_{u \in \mathbb{R}^m\times1} G(u),
\end{align}
subjected to the constraints
\begin{align*}u\ge0.\end{align*}
\begin{theorem}\label{thrm:lcp}
 Let $u^{n}$ and $f^{n}$ are $m\times 1$-column vectors achieved by splitting algorithm (*),
 \begin{align*}
 L^{0}_{\kappa}\sigma^{n+1}=f-(L^{-}_{\kappa}+L^{0}_{\kappa})u^{n}-L^{+}_{\kappa}u^{n+1},\\
 \sigma^{n+1}=\max\{0,\sigma^{n+1}\},\\
 u^{n+1}=u^{n}+\sigma^{n+1}\omega
\end{align*}
 then we have $u^{n} \rightarrow u$ and 
 $f^{n} \rightarrow f$ such that $u$ and $f$ is a solution of LCP problem.
\end{theorem}
\begin{proof}
For the proof of this theorem we refer to see Cryer \cite{Cryer}.
\end{proof}
The following error estimates are easily established for LCP problem for algorithm described above.
\begin{lemma}\label{lemma:lcp}
 Let $u$ is the exact solution of LCP problem define in Eqn.~\ref{eqnnew1}, also let $u^{n+1}$ is approximate solution
 obtained by the splitting of the form
  \begin{align*}
 L^{0}_{\kappa}\sigma^{n+1}=f-(L^{-}_{\kappa}+L^{0}_{\kappa})u^{n}-L^{+}_{\kappa}u^{n+1},\\
 \sigma^{n+1}=\max\{0,\sigma^{n+1}\},\\
 u^{n+1}=u^{n}+\sigma^{n+1}\omega
\end{align*}
Then following conditions hold
\begin{align*}
 \|u-u^{n+1}\|_{2} \le C_{2}\|u^{n+1}-u^{n}\|_{2} \\
  \|u-u^{n+1}\|_{1} \le C_{1}\|u^{n+1}-u^{n}\|_{1} \\
  \|u-u^{n+1}\|_{\infty} \le C_{\infty}\|u^{n+1}-u^{n}\|_{\infty}.
\end{align*}
\end{lemma}
\begin{proof}
Since From LCP problem we get
$$r_{\kappa}=L^{0}_{\kappa}u^{n}+f^{n}-(L^{-}_{\kappa}+L^{0}_{\kappa})u^{n}-L^{+}_{\kappa}u^{n+1} \ge 0 $$ 
 and $$r_{\kappa}^{+}=(r_{\kappa_{i,j}}^{+}),$$ where
 $$
r_{\kappa_{i,j}}^{+}=
\begin{cases}
r_{\kappa_{i,j}} \quad \text{if} \quad u^{n} >0\text{ and } u^{n+1} >0,\\
\min(0,r_{\kappa_{i,j}} )\quad \text{if} \quad u^{n}=0 \text{ and } u^{n+1} >0.
\end{cases}
$$
Now consider the following LCP 
$$L^{0}_{\kappa}u^{n+1} \le f-r_{\kappa_{i,j}}^{+},$$
$$u^{n+1}\ge 0,$$
$$u^{n+1}(L^{0}_{\kappa}u^{n+1} -f+r_{\kappa_{i,j}}^{+})=0$$
Now multiply $u^{T}$ in Eqn.~\ref{eqnnew1} and combing with equality term we get
$$ (u^{n+1}-u)^{T}L^{0}_{\kappa}u \le(u^{n+1}-u)^{T}f.$$
similar way we also get 
$$ (u-u^{n+1})^{T}L^{0}_{\kappa}u^{n+1} \le(u-u^{n+1})^{T}(f-r_{\kappa_{i,j}}^{+}).$$
Now by adding above two equations we get
$$  (u-u^{n+1})^{T}\nu_{*}(u-u^{n+1}) \le (u-u^{n+1})^{T}(-L_{\kappa}^{0})(u-u^{n+1}) \le (u-u^{n+1})^{T}( -r_{\kappa_{i,j}}^{+})$$
This implies that the following conditions hold
$$\| u-u^{n+1} \|_{1} \le \nu_{1}^{-1}\|-r_{\kappa_{i,j}}^{+} \|_{1},$$
$$\| u-u^{n+1} \|_{\infty} \le \nu_{\infty}^{-1}\|-r_{\kappa_{i,j}}^{+} \|_{\infty},$$
$$\| u-u^{n+1} \|_{2} \le \nu_{2}^{-1}\|-r_{\kappa_{i,j}}^{+} \|_{2}.$$
Now rest of the proof is followed from Lemma 2.2 mentioned in \cite{free_boundary}.
\end{proof}
Now we illustrate splitting for incompressible EHL model (we take $\rho, \eta$  and $\epsilon $ as constants here) in the form 
of inequalities as
\begin{example}\label{ex:three}
 \begin{align}
 \label{eqn55}
 (a(x,y)\mathscr{H}(u))_{x}-\epsilon \Delta u \ge f(x,y) \quad \forall x,y \in \Omega \nonumber\\
 u(x,y) \ge 0 \quad \forall x,y \in \Omega,\nonumber\\
 u(x,y)[(a(x,y)\mathscr{H}(u))_{x}-\epsilon \Delta u - f(x,y)]=0 \quad \forall x,y \in \Omega,\nonumber\\
 u(x,y) = g(x,y) \quad \forall x,y \in \partial \Omega,\nonumber\\
\mathscr{H}(u)=H_{00}+\frac{x^{2}+y^{2}}{2} + 
\frac{2}{\pi^{2}}\int_{-\infty}^{\infty} \int_{-\infty}^{\infty}\frac{u(x^{'},y^{'})dx^{'}dy^{'}}{\sqrt{(x-x^{'})^2+(y-y^{'})^2}}
  \end{align}
 \end{example}
For incompressible EHL problem $\kappa$-line distributive Jacobi splitting is written as 
consider the convection term of above Example~\ref{ex:three} as
\begin{align}
\label{eqn56}
 \frac{\partial h}{\partial x}=\frac{1}{h_{x}}\Big[(\mathscr{H}_{i,j}-\mathscr{H}_{i-1,j})
 -\frac{\kappa}{2}(\mathscr{H}_{i,j}-\mathscr{H}_{i-1,j})+\nonumber\\
 \frac{1+\kappa}{4}(\mathscr{H}_{i+1,j}-\mathscr{H}_{i,j})-\frac{1-\kappa}{4}(\mathscr{H}_{i-1,j}-\mathscr{H}_{i-2,j})\Big]
\end{align}
Now we will consider the following \underline{\bf Splitting : ${Ls4}$}
\begin{gather}
\label{eqn57}
 -\epsilon \Big[\Big\{ u_{i+1,j}+\sigma_{i+1}-\frac{(\sigma_{i} + \sigma_{i+2})}{4} \Big\}
 -\Big\{ u_{i,j}+\sigma_{i}-\frac{(\sigma_{i-1} + \sigma_{i+1})}{4} \Big\}\Big]\Big/h^{2}_{x} \nonumber\\
 -\epsilon \Big[\Big\{ u_{i-1,j}+\sigma_{i-1}-\frac{(\sigma_{i-2} + \sigma_{i})}{4} \Big\}
 -\Big\{ u_{i,j}+\sigma_{i}-\frac{(\sigma_{i-1} + \sigma_{i+1})}{4} \Big\}\Big]\Big/h^{2}_{x}\nonumber\\
 -\epsilon \Big[\Big\{ u_{i,j+1}-\frac{\sigma_{i}}{4} \Big\}
 -\Big\{ u_{i,j}+\sigma_{i}-\frac{(\sigma_{i-1} + \sigma_{i+1})}{4} \Big\}\Big]\Big/h^{2}_{x}\nonumber\\
 -\epsilon \Big[\Big\{ u_{i,j-1}-\frac{\sigma_{i}}{4} \Big\}
 -\Big\{ u_{i,j}+\sigma_{i}-\frac{(\sigma_{i-1} + \sigma_{i+1})}{4} \Big\}\Big]\Big/h^{2}_{x}\nonumber\\
-\frac{1}{h_{x}}\Big[\Big(\frac{2-\kappa}{2}\Big)
\Big(\sum_{k=i-1}^{i+1}\sigma \mathscr{G}_{ikjj}\sigma_{k}-\sum_{k=i-2}^{i}\sigma \mathscr{G}_{i-1kjj} \sigma_{k}\Big)\nonumber\\
 -\Big\{\frac{1+\kappa}{4}(\mathscr{H}_{i+1,j}-\mathscr{H}_{i,j})-\frac{1-\kappa}{4}(\mathscr{H}_{i-1,j}-\mathscr{H}_{i-2,j})\Big\}\Big]=f_{i,j}
\end{gather}
Another possibility is to consider the following splitting as 
\begin{align}
\label{eqn58}
 \frac{\partial h}{\partial x}=\frac{1}{h_{x}}\Big[(\mathscr{H}_{i,j}-\mathscr{H}_{i-1,j})
 -\frac{\kappa}{2}(\mathscr{H}_{i,j}-\mathscr{H}_{i-1,j})+\nonumber\\
 \frac{1+\kappa}{4}(\mathscr{H}_{i+1,j}-\mathscr{H}_{i,j})-\frac{1-\kappa}{4}(\mathscr{H}_{i-1,j}-\mathscr{H}_{i,j}
 +\mathscr{H}_{i,j}-\mathscr{H}_{i-2,j})\Big]
\end{align}
Hence overall equation is rewritten as \underline{\bf Splitting : ${Ls5}$}
\begin{gather}
\label{eqn59}
 -\epsilon \Big[\Big\{ u_{i+1,j}+\sigma_{i+1}-\frac{(\sigma_{i} + \sigma_{i+2})}{4} \Big\}
 -\Big\{ u_{i,j}+\sigma_{i}-\frac{(\sigma_{i-1} + \sigma_{i+1})}{4} \Big\}\Big]\Big/h^{2}_{x}\nonumber\\
 -\epsilon \Big[\Big\{ u_{i-1,j}+\sigma_{i-1}-\frac{(\sigma_{i-2} + \sigma_{i})}{4} \Big\}
 -\Big\{ u_{i,j}+\sigma_{i}-\frac{(\sigma_{i-1} + \sigma_{i+1})}{4} \Big\}\Big]\Big/h^{2}_{x}\nonumber\\
 -\epsilon \Big[\Big\{ u_{i,j+1}-\frac{\sigma_{i}}{4} \Big\}
 -\Big\{ u_{i,j}+\sigma_{i}-\frac{(\sigma_{i-1} + \sigma_{i+1})}{4} \Big\}\Big]\Big/h^{2}_{x}\nonumber\\
 -\epsilon \Big[\Big\{ u_{i,j-1}-\frac{\sigma_{i}}{4} \Big\}
 -\Big\{ u_{i,j}+\sigma_{i}-\frac{(\sigma_{i-1} + \sigma_{i+1})}{4} \Big\}\Big]\Big/h^{2}_{x}\nonumber\\
-\frac{1}{h_{x}}\Big[\Big(\frac{2-\kappa}{2}+\frac{1-\kappa}{4}\Big)
\Big(\sum_{k=i-1}^{i+1}\sigma \mathscr{G}_{ikjj}\sigma_{k}-\sum_{k=i-2}^{i}\sigma \mathscr{G}_{i-1kjj} \sigma_{k}\Big)\nonumber\\
 -\Big\{\frac{1+\kappa}{4}(\mathscr{H}_{i+1,j}-\mathscr{H}_{i,j})-\frac{1-\kappa}{4}(\mathscr{H}_{i,j}-\mathscr{H}_{i-2,j})\Big\}\Big]=f_{i,j}.
\end{gather}
More general discussion on convergence of these splittings are given in Section~\ref{sec:five}.
\section{TVD Implementation in Point Contact Model Problem}\label{sec:four}
In this Section, we implement the splitting discussed in the last Section~\ref{sec:three} and allow to extend it in EHL model.
A hybrid splitting presented here and it is determined by measuring the value of 
$$\min\Big(\frac{\epsilon(x,y)}{h_{x}},\frac{\epsilon(x,y)}{h_{y}}\Big).$$
This value is treated as switching parameter to perform two different splitting together while moving $x$ direction during the iteration. 
If the value $$\min\Big(\frac{\epsilon(x,y)}{h_{x}},\frac{\epsilon(x,y)}{h_{y}}\Big) > 0.6$$
then we apply line Gauss-Seidel splitting otherwise line Jacobi distributed splitting is incorporated in other words
\begin{align}
\label{eqn60}
L_{hs1}= \begin{cases} 
     L_{s1}\text{-splitting} &\text{ If } \min\Big(\frac{\epsilon(x,y)}{h_{x}},\frac{\epsilon(x,y)}{h_{y}}\Big) > 0.6 \\
     L_{s4}\text{-splitting} &\text{ If } \min\Big(\frac{\epsilon(x,y)}{h_{x}},\frac{\epsilon(x,y)}{h_{y}}\Big) \le 0.6.
   \end{cases}
\end{align}
\begin{align}
\label{eqn61} L_{hs2}= \begin{cases} 
     L_{s0}\text{-splitting} &\text{ If } \min\Big(\frac{\epsilon(x,y)}{h_{x}},\frac{\epsilon(x,y)}{h_{y}}\Big) > 0.6 \\
     L_{s5}\text{-splitting} &\text{ If } \min\Big(\frac{\epsilon(x,y)}{h_{x}},\frac{\epsilon(x,y)}{h_{y}}\Big) \le 0.6.
   \end{cases}
\end{align}
These constructions are well justified as the region where $\epsilon$ tends to zero, we end up having an ill-conditioned matrix system in the form 
of dense kernel matrix appear in film thickness term. Therefore, distributive Jacobi line splitting is implemented as a right pre-conditioner to reduce
the ill-conditioning of the matrix. However, in other part where $\epsilon$ is sufficiently large diffusion term dominates therefore we use Gauss line splitting. 
Considering the above setting in computational domain is quite demanding in EHL model as it allows us in reducing computational cost and storage issue.
We replace $\kappa$ value in splitting constructed in Section~\ref{sec:three} by incorporating appropriate limiter function $\phi$ there.
In next section, we define these two splitting in more general form having limiter function involve in the splitting.
\subsubsection{Limiter based Line Gauss-Seidel splitting}\label{subsec:num1}
EHL point contact problem is solved in the form of LCP and therefore in this Section we seek an efficient splitting for Reynolds equation
iterate along $x$-line direction to obtain the pressure solution. Now by using Theorem~\ref{thrm:lcp} and Lemma~\ref{lemma:lcp} we prove the convergence of the EHL solution.
This splitting is explained in the following way: First calculate updated pressure in $x$-line direction
as $\bar{u}_{i,j} = \tilde{u}_{i,j} + \sigma_{i}$ keeping $j$ fix  at a time for all $j$ in $y$-direction and then 
apply change $\sigma_{i}$ immediately to update the pressure $\tilde{u}$. The successive pressure change $\sigma_{i}$ along the $x$-direction can be calculated as below
\begin{align} 
\label{eqn62}
 \frac{\epsilon^{X}_{i+1/2,j}[({u}_{i+1,j} + \sigma_{i+1})-({u}_{i,j}+ \sigma_{i})]
 +\epsilon^{X}_{i-1/2,j}[({u}_{i-1,j} + \sigma_{i-1})-({u}_{i,j}+ \sigma_{i})]}{h_{x}} \nonumber \\
+  \frac{\epsilon^{Y}_{i,j+1/2}[{u}_{i,j+1}-({u}_{i,j}+ \sigma_{i})]
 +\epsilon^{Y}_{i,j-1/2}[{u}_{i,j-1}-({u}_{i,j}+ \sigma_{i})]}{h_{y}}   \nonumber \\
 -h_{y}((\rho \mathscr{H})^{*}_{i+1/2,j}-(\rho \mathscr{H})^{*}_{i-1/2,j}) = 0,
\end{align}
where terms read as
\small
\begin{align}
\label{eqn63}
 \epsilon^{X}_{i \pm 1/2,j}\stackrel{\text{defn}}{:=} h_{y}\epsilon_{i \pm 1/2,j} \nonumber,\quad
 \epsilon^{Y}_{i,j \pm 1/2}\stackrel{\text{defn}}{:=} h_{x}\epsilon_{i,j \pm 1/2} \nonumber, \nonumber\\
 \epsilon_{i \pm 1/2,j}\stackrel{\text{defn}}{:=} (\epsilon_{i,j}+\epsilon_{i\pm 1,j})/2,\quad
 \epsilon_{i,j \pm 1/2}\stackrel{\text{defn}}{:=} (\epsilon_{i,j}+\epsilon_{i,j\pm 1})/2,
\end{align}
where
\begin{gather*}
 \epsilon_{i ,j}=\frac{\rho(i,j) \mathscr{H}^{3}(i,j)}{\eta(i,j)\lambda}.
\end{gather*}
\begin{align}
\label{eqn64}
(\rho \mathscr{H})^{*}_{i+1/2,j}\stackrel{\text{def}}{:=}(\check{\rho} \bar{\mathscr{H}})_{i,j}+
\dfrac{1}{2}\phi(r_{i+1/2})((\check{\rho}\bar{\mathscr{H}})_{i+1,j}-(\check{\rho} \bar{\mathscr{H}})_{i,j})
\end{align}
\begin{align}
\label{eqn65}
(\rho \mathscr{H})^{*}_{i-1/2,j}\stackrel{\text{def}}{:=}(\check{\rho}\bar{\mathscr{H}})_{i-1,j}+
\dfrac{1}{2}\phi(r_{i-1/2})((\check{\rho} \bar{\mathscr{H}})_{i,j}-(\check{\rho} \bar{\mathscr{H}})_{i-1,j}),
\end{align}
\normalsize
where
\begin{align*}
r_{i+1/2} = \dfrac{(\check{\rho}\tilde{\mathscr{H}})_{i+1,j}-(\check{\rho}\tilde{\mathscr{H}})_{i,j}}
{(\check{\rho} \tilde{\mathscr{H}})_{i,j}-(\check{\rho} \tilde{\mathscr{H}})_{i-1,j}} \quad \text{and} \quad
r_{i-1/2}= \dfrac{(\check{\rho} \tilde{\mathscr{H}})_{i,j}-(\check{\rho} \tilde{\mathscr{H}})_{i-1,j}}
{(\check{\rho} \tilde{\mathscr{H}})_{i-1,j}-(\check{\rho} \tilde{\mathscr{H}})_{i-2,j}}.
\end{align*}
In above equation for each $i$,
\begin{align}
\label{eqn66}
 \bar{\mathscr{H}}_{i,j} = \tilde{\mathscr{H}}_{i,j} + \sum_{k}\mathscr{G}_{i,k,j,j}\sigma_{k}
\end{align}
It is observed that the magnitude of the kernel $\mathscr{G}_{i,k,j,j}$ in equation ~\ref{eqn66} diminishes rapidly as distance $|k-i|$
increase and therefore, we avoid unnecessary computation expense by allowing value of $k$ up to three terms.
So updated value of film thickness is rewritten as
\begin{align}
\label{eqn67}
 \bar{\mathscr{H}}_{i,j} = \tilde{\mathscr{H}}_{i,j} + \sum_{k=i-1}^{i+1}\mathscr{G}_{i,k,j,j}\sigma_{k}.
\end{align}
Hence, Eqn.~(\ref{eqn62}) is illustrated as
\begin{align}
\label{eqn68}
\mathcal{C}_{i+2,\phi}\sigma_{i+2} +\mathcal{C}_{i+1,\phi}\sigma_{i+1}+\mathcal{C}_{i,\phi}\sigma_{i}
+\mathcal{C}_{i-1,\phi}\sigma_{i-1}+\mathcal{C}_{i-2,\phi}\sigma_{i-2}= R_{i,j,\phi},
\end{align}
where $R_{i,j,\phi}$ and $\mathcal{C}_{i\pm.,\phi}$ are residual and coefficients of matrix arising due to linearized form involving the limiter function.
This setting leads to a band matrix formulation which is solved using Gaussian elimination with minimum computational work ($O(n)$).
\subsubsection{Limiter based Line-Distributed Jacobi splitting}\label{subsec:num2}
The understanding philosophy of line distributed Jacobi splitting is more physical than mathematical.
When diffusive coefficient tends to zero, pressure becomes large enough and non local effect of film thickness dominates in the region.
Therefore a small deflection in pressure change produces high error in updated film thickness eventually leads blow up the solution
after few iterations. This numerical instability is overcome by interacting with the neighborhood points during iteration.    
During this process the computed change of pressure at one point of the line are shared to its neighbor cells.
In other words, a given point of a line new pressure $\bar{u}_{i,j}$ is computed from the summation of the changes
coming from neighboring points plus the old approximated pressure $\tilde{u}_{i,j}$
\begin{align}
\label{eqn69}
\bar{{u}}_{i,j} = \tilde{{u}}_{i,j}+\sigma_{i,j}-\dfrac{(\sigma_{i+1,j}+\sigma_{i-1,j}+\sigma_{i,j+1}+\sigma_{i,j-1})}{4}
\end{align}
In this case, changes are incorporated only at the end of a complete iteration sweep.
Therefore, overall splitting is derived as below
\begin{align}
\label{eqn70}
  \frac{\epsilon^{X}_{i+1/2,j}[({u}_{i+1,j} + \sigma_{i+1}-\dfrac{(\sigma_{i}+\sigma_{i+2})}{4})  
  -({u}_{i,j}+ \sigma_{i}-\dfrac{(\sigma_{i-1}+\sigma_{i+1})}{4})]}{h_{x}} \nonumber \\
 +\frac{\epsilon^{X}_{i-1/2,j}[({u}_{i-1,j} + \sigma_{i-1}-\dfrac{(\sigma_{i-2}+\sigma_{i})}{4})
 -({u}_{i,j}+ \sigma_{i}-\dfrac{(\sigma_{i-1}+\sigma_{i+1})}{4})]}{h_{x}} \nonumber \\
 +\frac{\epsilon^{Y}_{i,j+1/2}[{u}_{i,j+1}-\dfrac{\sigma_{i}}{4}-({u}_{i,j}+ \sigma_{i}-\dfrac{(\sigma_{i-1}+\sigma_{i+1})}{4})]}{h_{y}}+ \nonumber \\
 \frac{\epsilon^{Y}_{i,j-1/2}[{u}_{i,j-1}-\dfrac{\sigma_{i}}{4}-({u}_{i,j}+ \sigma_{i}-\dfrac{(\sigma_{i-1}+\sigma_{i+1})}{4})]}{h_{y}}   \nonumber \\
 -h_{y}((\rho \mathscr{H})^{*}_{i+1/2,j}-(\rho \mathscr{H})^{*}_{i-1/2,j}) = 0.
\end{align}
The following notion used in Eqn.~\ref{eqn70} defined as
\begin{align}
\label{eqn71}
 \epsilon^{X}_{i \pm 1/2,j}\stackrel{\text{defn}}{:=} h_{y}\epsilon_{i \pm 1/2,j} \nonumber \\
 \epsilon^{Y}_{i,j \pm 1/2}\stackrel{\text{defn}}{:=} h_{x}\epsilon_{i,j \pm 1/2}
\end{align}
\begin{gather*}
 \epsilon_{i \pm 1/2,j} = 0.5\Big(\frac{\rho(i\pm 1,j) \mathscr{H}^{3}(i \pm 1,j)}{\eta(i \pm 1,j)\lambda}
 +\frac{\rho(i \pm 1,j) \mathscr{H}^{3}(i \pm 1,j)}{\eta(i \pm 1,j)\lambda}\Big),\\
\epsilon_{i,j \pm 1/2} = 0.5\Big(\frac{\rho(i,j\pm 1) \mathscr{H}^{3}(i,j \pm 1)}{\eta(i,j \pm 1)\lambda}
 +\frac{\rho(i,j \pm 1) \mathscr{H}^{3}(i ,j \pm 1)}{\eta(i \pm 1,j \pm 1)\lambda}\Big). 
\end{gather*}
\begin{align}
\label{eqn72}
(\rho \mathscr{H})^{*}_{i+1/2,j}\stackrel{\text{def}}{:=}(\check{\rho} \bar{\mathscr{H}})_{i,j}+
\dfrac{1}{2}\phi(r_{i+1/2})((\check{\rho}\bar{\mathscr{\mathscr{H}}})_{i+1,j}-(\check{\rho} \bar{\mathscr{H}})_{i,j})
\end{align}
\begin{align}
\label{eqn73}
(\rho \mathscr{H})^{*}_{i-1/2,j}\stackrel{\text{def}}{:=}(\check{\rho}\bar{\mathscr{H}})_{i-1,j}+
\dfrac{1}{2}\phi(r_{i-1/2})((\check{\rho} \bar{\mathscr{H}})_{i,j}-(\check{\rho} \bar{\mathscr{H}})_{i-1,j}),
\end{align}
where
\begin{align*}
r_{i+1/2} = \dfrac{(\check{\rho}\tilde{\mathscr{H}})_{i+1,j}-(\check{\rho}\tilde{\mathscr{H}})_{i,j}}
{(\check{\rho} \tilde{\mathscr{H}})_{i,j}-(\check{\rho} \tilde{\mathscr{H}})_{i-1,j}} \quad \text{and} \quad
r_{i-1/2}= \dfrac{(\check{\rho} \tilde{\mathscr{H}})_{i,j}-(\check{\rho} \tilde{\mathscr{H}})_{i-1,j}}
{(\check{\rho} \tilde{\mathscr{H}})_{i-1,j}-(\check{\rho} \tilde{\mathscr{H}})_{i-2,j}}.
\end{align*}
In above equation, discretization of convection term defined same as Line Gauss-Seidel relaxation case.
However, due to distributive change of the pressure, the updated value of film thickness is described as
\begin{align}
\label{eqn74}
 \bar{\mathscr{H}}_{i,j} = \tilde{\mathscr{H}}_{i,j} + \sum_{k}\sigma \mathscr{G}_{i,k,j,j}\sigma_{k}, 
\end{align}
where $$\sigma \mathscr{G}_{i,i,j,j} = \mathscr{G}_{i,i,j,j}-(\mathscr{G}_{i,i-1,j,j}+\mathscr{G}_{i,i+1,j,j}
+\mathscr{G}_{i,i,j,j-1}+\mathscr{G}_{i,i,j,j+1}).$$
After few manipulation of Eqn.~\ref{eqn70}, we get system of band matrix which is solved using Gaussian elimination approach.\\
The force balance equation is incorporated in our numerical calculation by updating the constant value $\mathscr{H}_{00}$.
The updated value of $\mathscr{H}_{00}$ is performed according to
\begin{align}
\label{eqn75}
 \mathscr{H}_{00} \leftarrow \mathscr{H}_{00}-c\Big( \frac{2\pi}{3}
 -h_{x}h_{y}\sum_{i=1}^{n_{x}} \sum_{j=1}^{n_{y}} {u}_{i,j} \Big),
\end{align}
where $c$ is a relaxation parameter having range between $0.01-0.1$.
\vspace{-0.4cm}
\section{Fourier Analysis}\label{sec:five}
Performance and asymptotic estimate of above splitting is measured through the Fourier analysis by considering infinite grid 
\begin{align}
\label{eqn76}
{\mathbb{G}^{f}}_{h}:= \{ {\bf x} = (\xi_{1}h,\xi_{2}h) : \xi =(\xi_{1},\xi_{2}) \in \mathbb{Z}\times \mathbb{Z}\}
\end{align}
and infinite grid function defined on ${\mathbb{G}^{f}}_{h}$ by the linear span of the Fourier components 
\[\mathbb{T} ^{h}=\text{span}\Big\{\varphi (\theta,{\bf x})=e^{i(\xi_{1}\theta_{1}+\xi_{2}\theta_{2})}:
\theta=(\theta_{1},\theta_{2}) \in (-\pi,\pi]^{2},{\bf x}\in{G^{f}}_{h} \Big\}.\]
These basis functions $e^{i{\xi\theta}} \in \mathbb{T} ^{h} $ are orthogonal with respect to the inner product
\begin{align}
\label{eqn77}
 \langle u_{h},v_{h}\rangle:= \lim_{l \to \infty}\frac{1}{4 l^2}\sum_{|{\bf \xi}| \le l} u_{h}(\xi_{1}h,\xi_{2}h)
 \overline{v_{h}(\xi_{1}h,\xi_{2}h)},
\end{align}
where $ u_{h},v_{h} \in  \mathbb{T} ^{h}$. Furthermore, we will define orthogonal space to identity function 
$\mathbb{I} \in \mathbb{T} ^{h}  $ as 
\begin{align}
\label{eqn78}
 \mathbb{T} ^{h}_{\bot} =\{ v_{h}: \langle \mathbb{I},v_{h}\rangle=0 \}
\end{align}
Moreover, discrete solution $u_{h}$ is described as Fourier transform $\hat{u}$ a
linear combinations of the basis functions $e^{i\xi\theta} \in \mathbb{T}^{h}$
\begin{align}
\label{eqn79}
u_{h}= \lim_{l \to \infty}\frac{1}{2l}\sum_{|\xi| \le l} \hat{u}_{h}(\xi)e^{i\xi \theta}.
\end{align}
The Fourier space $\mathbb{T}^{h}$ is illustrated as four-dimensional subspaces
 \[\mathbb{T}^{h}_{\theta} = \text{span}\{\varphi(\theta^{\alpha_{1}\alpha_{2}},{\bf x})= e^{i{\bf k}
 \theta^{\alpha_{1}\alpha_{2}}}; \alpha_{1},\alpha_{2} \in\{0,1\} \},
 \text{ where }{\bf x} \in {\mathbb{G}^{f}}_{h};\]
 \[\theta^{00} \in (-\pi/2,\pi/2]^2,
 \theta^{\alpha_{1}\alpha_{2}}= (\theta_{1}-\alpha_{1}\text{sign}(\theta_{1})\pi,\theta_{2}-\alpha_{2}\text{sign}(\theta_{2})\pi).\]
We say discretized PDE of the form
   \begin{align}
   \label{eqn80}
    L_{h}u_{h} = f_{h}
   \end{align}
is solvable if $f_{h} \in \mathbb{T} ^{h}_{\bot}$. Moreover, solution will be unique if $u_{h} \in \mathbb{T} ^{h}_{\bot}$.
Let relaxation method defined via operator splitting as
   \begin{align}
   \label{eqn81}
    L_{h}^{+}\bar{u}_{h}+L_{h}^{-}\tilde{u}_{h} = f_{h},
   \end{align}
where $\tilde{u}_{h}$ and $\bar{u}_{h}$ are old and updated approximation to the solution $u_{h}$.
Now we are interested in constructing a splitting which reduce our computed error significantly. Such behavior is investigated
by measuring error equation as   
\begin{align}
\label{eqn82}
 \bar{e}_{h}=\mathcal{S}_{h}\tilde{e}_{h},
\end{align}
where $\tilde{e}_{h}=u_{h}-\tilde{u}_{h}$, $\bar{e}_{h}=u_{h}-\bar{u}_{h}$
and $\mathcal{S}_{h} :=-(L^{+}_{h})^{-1}L^{-}_{h}$.
Now apply Fourier transform in above equation for $\hat{L}^{+}_{h}(\theta)\neq 0$ we have following relation
 \begin{align}
 \label{eqn83}
  \mathcal{S}_{h}\varphi(\theta,{\bf x}) = \hat{\mathcal{S}}_{h}(\theta)\varphi(\theta,{\bf x}) 
  \quad \forall {\bf x} \in \mathbb{G}^{f}_{h},
 \end{align}
and smoothing factor notation as 
\begin{align}
\label{eqn84}
 \mu_{1}(\mathcal{S}_{h}):=sup\{|\hat{\mathcal{S}}_{h}(\theta)|: \theta \in \Theta_{high} \},
\end{align}
where $\hat{\mathcal{S}}_{h}(\theta):=-\hat{L}^{-}_{h}(\theta)/\hat{L}^{+}_{h}(\theta)$.
\subsubsection{Fourier analysis of $\kappa$ splitting}
Let $ \tilde{u}^{h}_{i,j} $ current updated to the solution for given $j$ line we are solving equations. 
For given $j$ a new updated $\bar{u}^{h}_{i,j}$ for all $i$ of that line according to
\begin{gather}
\label{eqn102}
\Big \{ -\epsilon \frac{\bar{u}_{i-1,j}-2\bar{u}_{i,j}+\bar{u}_{i+1,j}}{h_{x}^{2}}\Big\}
 +\Big \{ -\epsilon \frac{\bar{u}_{i,j-1}-2\bar{u}_{i,j}+\tilde{u}_{i,j+1}}{h_{y}^{2}}\Big\} \nonumber\\
 +\frac{a}{h}\Big \{(\bar{u}_{i,j}-\bar{u}_{i-1,j})-\frac{\kappa}{2}(\bar{u}_{i,j}-\bar{u}_{i-1,j})
 +\frac{1-\kappa}{4}(\bar{u}_{i,j}-\bar{u}_{i-1,j})\nonumber\\
 +\frac{1+\kappa}{4}(\tilde{u}_{i+1,j}-\tilde{u}_{i,j})-\frac{1-\kappa}{4}(\tilde{u}_{i,j}-\tilde{u}_{i-2,j}) \Big\}
 =f_{i,j},   
 \end{gather}
for $2 \le i \le (n_{x}-1)$ and for given value $j$ such that $1 \le j \le n_{y}-1$ holds.
During Gauss-Seidel line relaxation, we will use previously computed new solution of line $j-1$ in our next new updated solution  
of line $j$. Hence error equation is written as 
 \begin{gather}
 \label{eqn103}
  -\Big\{ \frac{\epsilon}{h^2}+\frac{a(1.25-0.75\kappa)}{h}\Big\}\bar{e}_{i-1,j}
  +\Big\{\frac{4\epsilon}{h^2} + \frac{a(1.25-0.75\kappa)}{h} \Big\}\bar{e}_{i,j}\nonumber\\
  -\Big\{\frac{\epsilon}{h^2} \Big\}\bar{e}_{i+1,j}-\Big\{ \frac{\epsilon}{h^2}\Big\}\bar{e}_{i,j-1}
  -\Big\{\frac{\epsilon}{h^2} \Big\}\tilde{e}_{i,j+1}\nonumber\\
  +\Big\{\frac{a(1+\kappa)}{4h}\Big\}(\tilde{e}_{i+1,j}-\tilde{e}_{i,j})
  -\Big\{\frac{a(1-\kappa)}{4h} \Big\}(\tilde{e}_{i,j}-\tilde{e}_{i-2,j})=0
 \end{gather}
 and $\kappa$-smoothing factor is denoted as 
  \begin{gather}
  \label{eqn104}
  |\mathcal{S}_{h}^{\kappa}(\theta_{1},\theta_{2})|= \nonumber\\
 \Big|\frac{\alpha_{1}e^{i\theta_{2}}+0.25\beta(1+\kappa)(e^{i\theta_{1}}-1)-0.25\beta(1-\kappa)(1-e^{-i2\theta_{1}})}
{(-\alpha_{1}-\beta(1.25-0.75\kappa))e^{-i\theta_{1}}+4\alpha_{1}+\beta(1.25-0.75\kappa)-\alpha_{1}(e^{i\theta_{1}}
+e^{-i\theta_{2}})}\Big|,
  \end{gather}
 where $\alpha_{1} = \epsilon/h^{2}$ and $\beta=a/h$. 
Smoothing factor plot is given in Fig.~\ref{fig:lfa}
       \begin{figure}
        \centering
        \includegraphics[width=10cm,height=10cm,angle =-90,keepaspectratio]{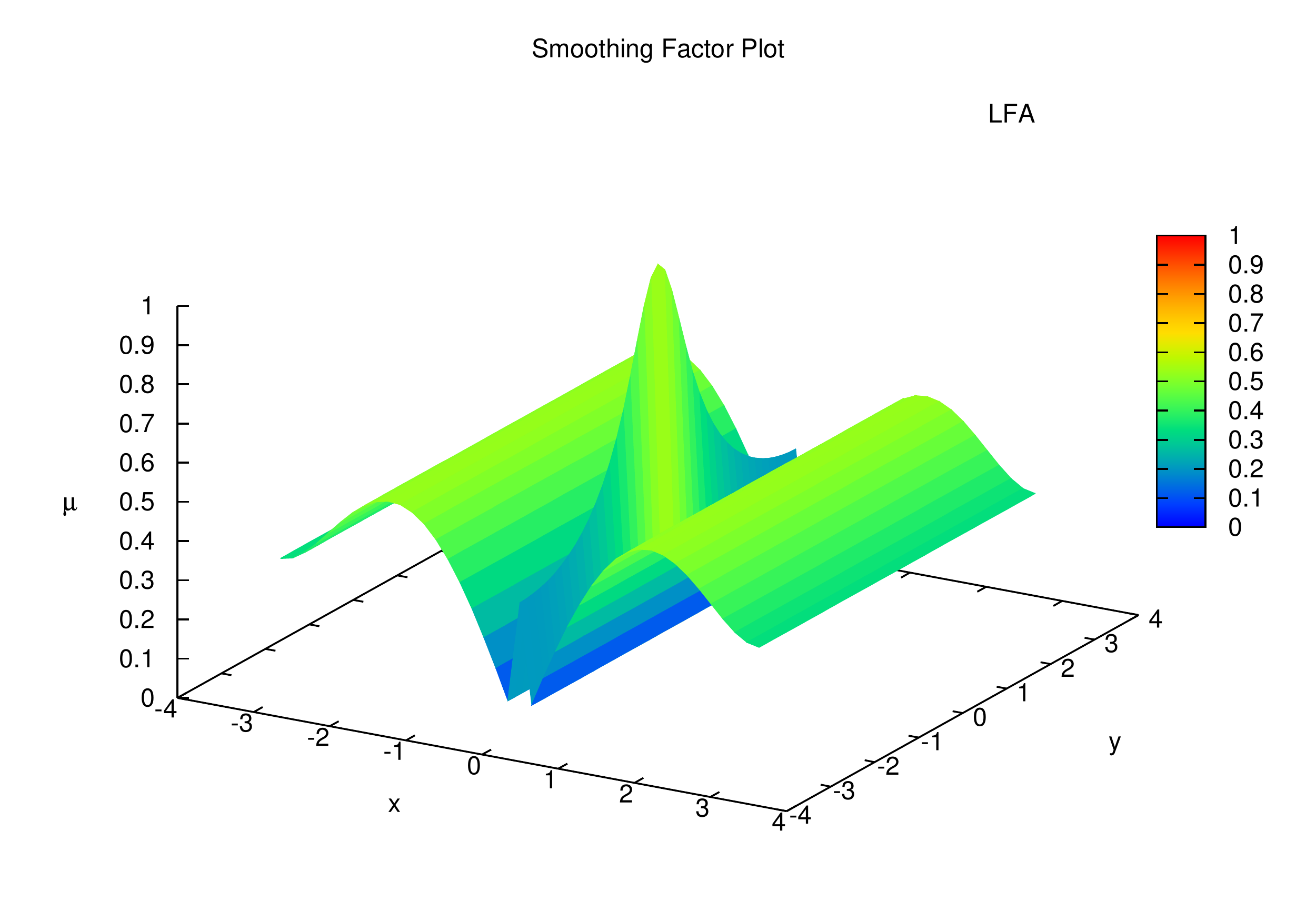}
        \caption{Smoothing factor of example 1 ( see Eqn .~\ref{eqn39}) using splitting 
        $L_{s0}$ for value $\epsilon=10^{-6}$, $\kappa=1/3$, $h=1/64$.}
        \label{fig:lfa}
    \end{figure}
Two grid iteration matrix is written as
\begin{align}
\label{eqn85}
C^{2h}_{h}=I_{h}-P^{h}_{2h}(L_{2h})^{-1}R^{2h}_{h}L_{h}
\end{align}
and two grid error equation is defined as
\begin{align}
\label{eqn86}
 e^{\text{new}} = \mathcal{S}^{\nu_{2}} C^{2h}_{h}\mathcal{S}^{\nu_{1}}e^{\text{old}}
 =\mathcal{M}^{2h}_{h}e^{\text{old}}.
\end{align}
Here by multiplying $C_{h}^{2h}$ to the space $\mathbb{T}^{h}_{\theta}$, where 
${\theta} \in \tilde{\varTheta}^{00} =\varTheta^{00}- \{ {\theta}:L_{2h}(2\theta^{00})=0\}$
leaves the space invariant.\\
\begin{gather}
\label{eqn87}
C^{2h}_{h}:\mathbb{T}_{\theta}^{h}\longrightarrow \mathbb{T}_{\theta}^{h}.
\end{gather}
Fourier representation of two grid is performed in following way
\begin{gather}
 L_{h}:\mathbb{T}_{\theta}^{h}\longrightarrow \mathbb{T}_{\theta}^{h}, \quad 
 L_{2h}:\mathbb{T}_{\theta}^{2h}\longrightarrow \mathbb{T}_{\theta}^{2h} \\
 R_{h}:\mathbb{T}_{\theta}^{h}\longrightarrow \mathbb{T}_{\theta}^{2h},
 \quad P_{h}:\mathbb{T}_{\theta}^{2h} \longrightarrow \mathbb{T}_{\theta}^{h}
 \quad \text{with} \quad {\theta} \in \tilde{\varTheta}^{00}\\
 \mathcal{S}:\mathbb{T}_{\theta}^{h}\longrightarrow \mathbb{T}_{\theta}^{h}  
({\theta} \in \tilde{\varTheta}^{00})
\end{gather}
Spectral radius is computed in the following way
\begin{align}
\label{eqn91}
 \rho^{*} = \rho(\mathcal{M}^{2h}_{h}) = \sup_{{\theta} \in \tilde{\varTheta}^{00}}\rho(\mathcal{M}^{2h}_{h}({\theta}))=
 \sup_{{\theta} \in \tilde{\varTheta}^{00}}\rho({\theta}),
\end{align}
where 
\begin{gather}
\label{eqn92}
\tilde{\mathcal{M}}^{2h}_{h}({\theta})=
\tilde{\mathcal{S}}^{\nu_{2}}(I_{h}
-\tilde{P}_{2h}^{h}(\tilde{L}_{2h})^{-1}\tilde{R}_{h}^{2h}\tilde{L}_{h})\tilde{\mathcal{S}}^{\nu_{1}}
,\tilde{\mathcal{M}}^{2h}_{h}({\theta})=\mathcal{M}^{2h}_{h}|_{\mathbb{T}_{\theta}^{h}}({\theta} \in \tilde{\varTheta}^{00}).
\end{gather}
The Fourier symbols of the multi-grid operators for each harmonic in $\mathbb{T}^{h}_{\theta}$ is calculated as follows:
\begin{align}
 \tilde{S}^{\nu} = 
 \begin{pmatrix}
  \mu({\theta^{00}}) &  &  &  \\
  & \mu({\theta^{10}}) &  &  \\
   &   & \mu({ \theta^{01}})  &  \\
   &  &  &  \mu({\theta^{11}}) 
 \end{pmatrix}^{\nu},\\
 \tilde{L}_{h} = 
 \begin{pmatrix}
  \tilde{L}_{h}({\theta^{00}}) &  &  &  \\
  & \tilde{L}_{h}({\theta^{10}}) &  &  \\
   &   & \tilde{L}_{h}({\theta^{01}})  &  \\
   &  &  &  \tilde{L}_{h}({\theta^{11}}) 
 \end{pmatrix},\\ 
 \tilde{R}_{h}=(\tilde{R}_{h}({\theta^{00}}),\tilde{R}_{h}({\theta^{10}}),
 \tilde{R}_{h}({\theta^{01}}),\tilde{R}_{h}({\theta^{11}})),\\
 \tilde{P}_{h}=(\tilde{P}_{h}({\theta^{00}}),\tilde{P}_{h}({\theta^{10}}),
 \tilde{P}_{h}({\theta^{01}}),\tilde{P}_{h}({\theta^{11}}))^{T},\\
 \tilde{L}_{2h} = \tilde{L}_{2h}( 2 {\theta^{00}} ) 
\end{align}
For the transfer operators
\begin{align}
 \tilde{L}_{h}({\theta}^{**}) = \sum_{\mu_{x} \in J}\sum_{\mu_{y} \in J}a^{h(2)}_{\mu_{x}\mu_{y}}
 e^{i\theta^{**}_{x}\mu_{x}}e^{i\theta^{**}_{y}\mu_{y}} \\
 \tilde{L}_{2h}(2{\theta}^{00}) = \sum_{\mu_{x} \in J}\sum_{\mu_{y} \in J}a^{2h(2)}_{\mu_{x}\mu_{y}}
 e^{i\theta^{00}_{x}\mu_{x}}e^{i\theta^{00}_{y}\mu_{y}}
\end{align}
Since we can always get a nonsingular matrix $P$ same order as $C^{2h}_{h}$ such that $PC^{2h}_{h}P^{-1}=Q^{2h}_{h}$ holds, where
$Q_{h}^{2h}$ a block matrix consisting of $4 \times 4$ diagonal block $\tilde{Q}^{2h}_{h}(\theta)$
looks for all $\theta \in \tilde{\Theta}_{00}$ like
\begin{align}
\label{eqn2gm}
 \tilde{Q}^{2h}_{h}= \begin{pmatrix}
0\\
&\!\!1\\
&&1\\
&&&1
\end{pmatrix}
\end{align}
then the smoothing factor is equivalent to
\begin{align}
\label{eqn101}
 \mu= \sup_{\theta \in \tilde{\Theta}_{00}} \rho(\tilde{S}(\theta)Q_{h}^{2h}(\theta))=
 \sup_{\theta \in \tilde{\Theta}_{00}} \rho(\theta)
\end{align}
Computation of $\mu$ is important for observing two-grid convergence during relaxation.
In next Section we illustrate a criterion for two-grid convergence.
\subsection{Convergence criterion of hybrid splitting}
In this section, we give a general criteria for the convergence study of hybrid schemes used in our EHL model problem.
Let us reconsider linear system
$$ L_{\kappa}u=f,$$
where $[L_{\kappa}]_{m\times m}$ a regular matrix (for definition see \cite{Varga}) and $f$ and $u$ are known values.
For applying hybrid splitting in above equation matrix $L_{\kappa}$ is understood as 
$$ L_{\kappa}=L_{\kappa}^{\Omega_{\epsilon}}L_{\kappa}^{\Omega'_{\epsilon}},$$
where $[L_{\kappa}^{\Omega_{\epsilon}}]$ and  $[L_{\kappa}^{\Omega'_{\epsilon}}]$ are regular applied splittings in 
$$\Omega_{\epsilon}=\Big\{(x,y)\Big|\min\Big(\frac{\epsilon(x,y)}{h_{x}},\frac{\epsilon(x,y)}{h_{y}}\Big)\le 0.6\Big\}$$ 
and 
$$\Omega'_{\epsilon}=\Big\{(x,y)\Big|\min\Big(\frac{\epsilon(x,y)}{h_{x}},\frac{\epsilon(x,y)}{h_{y}}\Big)> 0.6\Big\}$$
sub-domains respectively.\\
Now assume that $[L_{\kappa}^{\Omega_{\epsilon}}]$ has the following splitting
$$L_{\kappa}^{\Omega_{\epsilon}}=M_{\kappa}^{\Omega_{\epsilon}}-N_{\kappa}^{\Omega_{\epsilon}},$$
where $M_{\kappa}^{\Omega_{\epsilon}}$ is a regular easily invertible matrix and $N_{\kappa}^{\Omega_{\epsilon}}$ 
is a positive rest matrix. Then our splitting can be defined as 
$$ u^{n+1}_{\Omega_{\epsilon}}=u^{n}_{\Omega_{\epsilon}}-(M_{\kappa}^{\Omega_{\epsilon}})^{-1}(L_{\kappa}^{\Omega_{\epsilon}}-f)$$
Then above iteration will converge for any initial guess $u^{0}$ if following theorem holds 
\begin{theorem}
 Let $L_{\kappa}^{\Omega_{\epsilon}}=M_{\kappa}^{\Omega_{\epsilon}}-N_{\kappa}^{\Omega_{\epsilon}}$ be a regular splitting of matrix $L_{\kappa}^{\Omega_{\epsilon}}$
 and $(L_{\kappa}^{\Omega_{\epsilon}})^{-1} \ge 0$, then we have 
 $$\rho((M_{\kappa}^{\Omega_{\epsilon}})^{-1}N_{\kappa}^{\Omega_{\epsilon}})
 =\frac{\rho((L_{\kappa}^{\Omega_{\epsilon}})^{-1}N_{\kappa}^{\Omega_{\epsilon}})}
 {1+\rho((L_{\kappa}^{\Omega_{\epsilon}})^{-1}N_{\kappa}^{\Omega_{\epsilon}})} < 1$$
\end{theorem}
\begin{proof}
For the proof of this theorem we refer to see Varga \cite{Varga}.
\end{proof}
Now we will prove other part of matrix splitting $L_{\kappa}^{\Omega'_{\epsilon}}$. This part of matrix there is no straightforward 
splitting is available (see \cite{Varga,wittum}). Let $L_{\kappa}^{\Omega'_{\epsilon}}$ is regular,
but dense and the designing suitable splitting in the sense of Varga is complicated.
Suppose if it is possible to construct nonsingular matrix $L^{r}_{\kappa}$  such that equation below
$$L_{\kappa}^{\Omega'_{\epsilon}}L^{r}_{\kappa}=M_{\kappa}^{\Omega'_{\epsilon}}-N_{\kappa}^{\Omega'_{\epsilon}} $$
is easy to solve and we can rewrite splitting as 
$$L_{\kappa}^{\Omega'_{\epsilon}}=(M_{\kappa}^{\Omega'_{\epsilon}}-N_{\kappa}^{\Omega'_{\epsilon}}){L^{r}_{\kappa}}^{-1} $$
Then for above splitting our iteration is denoted as
$$u^{n+1}=u^{n}-L^{r}_{\kappa}(M_{\kappa}^{\Omega'_{\epsilon}})^{-1}(L_{\kappa}^{\Omega'_{\epsilon}}-f)$$
Therefore above iteration will converge for any initial guess if following theorem holds
\begin{theorem}
 Let $(M_{\kappa}^{\Omega'_{\epsilon}}-N_{\kappa}^{\Omega'_{\epsilon}})(L^{r}_{\kappa})^{-1}$ be a regular splitting of matrix
 $L_{\kappa}^{\Omega'_{\epsilon}}$
 and $(L_{\kappa}^{\Omega'_{\epsilon}})^{-1} \ge 0$, then we have 
 $$\rho(L^{r}_{\kappa}(M_{\kappa}^{\Omega'_{\epsilon}})^{-1}N_{\kappa}^{\Omega'_{\epsilon}}(L^{r}_{\kappa})^{-1})
 =\frac{\rho((L_{\kappa}^{\Omega'_{\epsilon}})^{-1}N_{\kappa}^{\Omega'_{\epsilon}}(L^{r}_{\kappa})^{-1})}
 {1+\rho((L_{\kappa}^{\Omega'_{\epsilon}})^{-1}N_{\kappa}^{\Omega'_{\epsilon}}(L^{r}_{\kappa})^{-1})} < 1$$ 
\end{theorem}
The following theorem providing sufficient conditions for the convergence of the two-grid method $Q_2$ ( define in Eqn.\ref{eqn2gm}) is due to Hackbusch.
\begin{theorem}
Let us assume that $\mathcal{S}_{l}$ is a smoothing operator for $K_l$ that means there exist $\eta(\nu)$ and $\nu'(h)$
so that the following condition holds
 \begin{align}
 \label{eqn105}
  ||K_{l}\mathcal{S}_{l}^{\nu}||_{F\leftarrow U} \le \eta(\nu) \quad \quad \forall \quad \nu : 1\le\nu\le\nu'(h), \quad l\ge 2, \nonumber \\  
  \eta(\nu) \rightarrow 0 \quad \text{for} \quad \nu \rightarrow \infty, \quad \nu '(h) =\infty \quad \text{or} \quad 
  \nu'(h) \rightarrow \infty \quad \text{for} \quad h \rightarrow 0
 \end{align}
and also assume that operator $K_l$ is approximated accurately (by prolongation and restriction operator) in the following sense such that
     $\exists \quad C_{A} \rightarrow 0$, independent of $h$ so that 
 \begin{align}
 \label{eqn106}
  ||K^{-1}_{l}-P(K_{l-1})^{-1}R||_{U\leftarrow F} \le C_{A}  \quad \forall \quad  l \ge 2
 \end{align}
  then there exist $h$ and $\nu \in \mathbf{N}$:
  \begin{align}
   \label{eqn107}
  ||Q_{2,l}(\nu,0)||_{U \leftarrow U} \le C_{A} \eta(\nu) < 1 
  \end{align}
holds for $\nu$ with $\nu'(h_{l}) \ge \nu \ge \nu(h_{l})$ and $h_{2} \le h$ and  the two-grid method $Q_{2,l}$ from 
Eqn.~\ref{eqn92} converges monotonically, independently of $h$.
\end{theorem}
\begin{proof}
 It follows straight way by taking \small $Q_{2,l}(\nu,0) = (K^{-1}_{l}-P(K_{l-1})^{-1}R)(K_{l}\mathcal{S}_{l}^{\nu})$\normalsize.
\end{proof}
\section{Numerical Results}\label{sec:six}
In Section~\ref{sec:three}, we have illustrated TVD implementation for solving linear convection-diffusion problem through a class of splittings.
Now we investigate the performance of mentioned splittings and compare the results with classical defect-correction.
For numerical tests we consider analytical solution as $u=x^{4}+y^{4}$ from Oosterlee \cite{Oosterlee}.
All numerical computations is performed on author's personal laptop having 2GB RAM and Intel(R) Core(TM) i3-2328M CPU @ 2.20GHz. 
Dirichlet boundary is imposed for all test cases on domain $\Omega=\Big\{ (x,y); -1 \le x \le 1,-1 \le y \le 1 \Big\}$. 
For all numerical experiments, we take diffusion coefficient $\epsilon=10^{-6}$ and $\kappa=-1.0,0.0,1/3$.
Numerical tests are performed for the problem given as Example~\ref{ex:one} using $Ls0$ splitting, $Ls1$ splitting and classical defect-correction
technique using hierarchical multi-level grid. Computational results of relative error and corresponding order
in $L^{1},L^{\infty},L^{2}$-norms are presented on Table~\ref{Table.l1sp1}-~\ref{Table.l6sp0} on the finest grid level ($7^{th}$ level using $3 V(2,1)$ cycle).  
       \begin{figure}
        \centering
        \includegraphics[width=11cm,height=11cm,angle =-90,keepaspectratio]{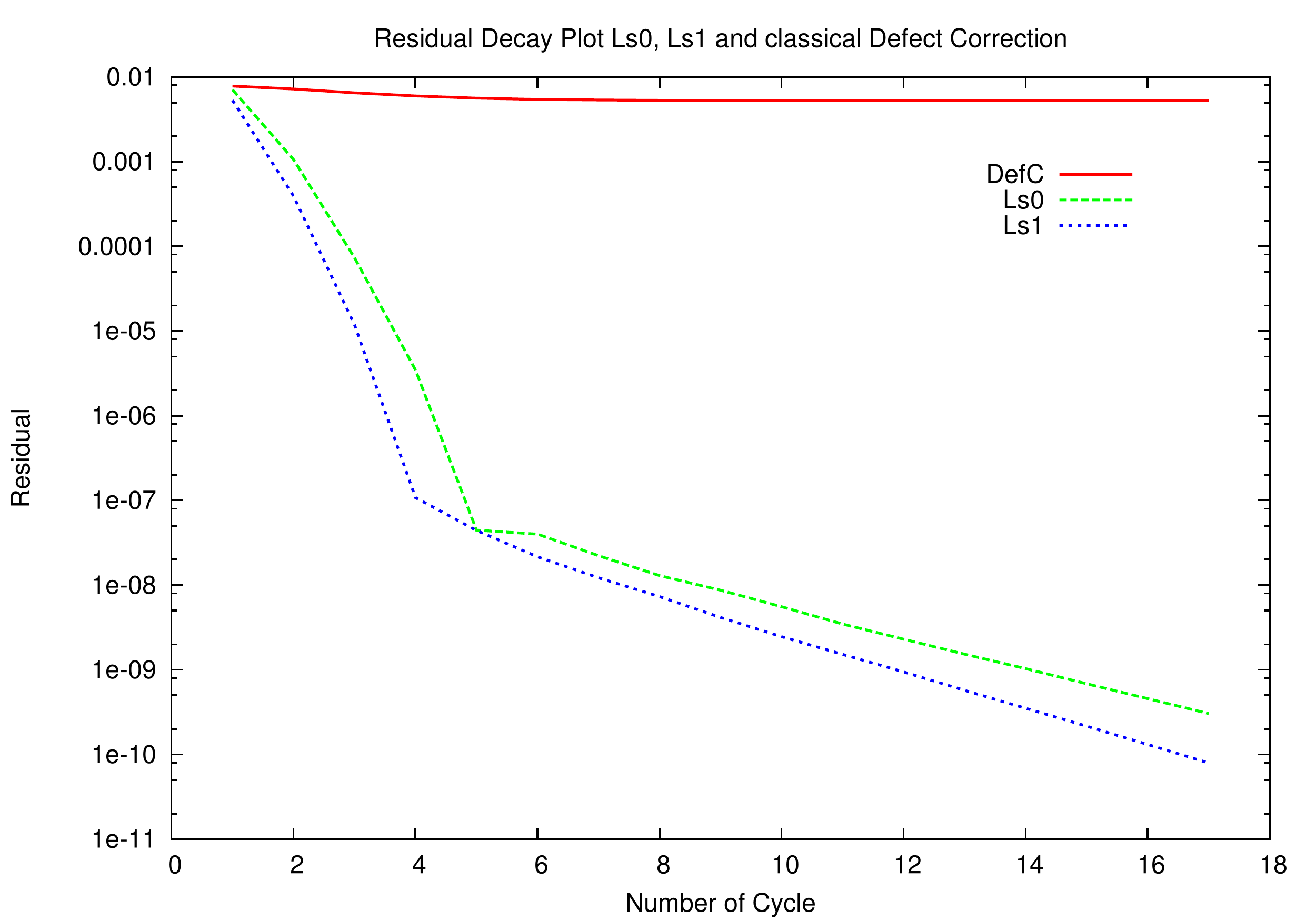}
        \caption{Comparison of residual decay of splitting $Ls0$ and splitting $Ls1$ 
        for $\kappa=1/3$ on $7^{th}$ level $V(2,1)$ cycle}
        \label{figg5}
    \end{figure}
           \begin{figure}
        \centering
        \includegraphics[width=11cm,height=11cm,angle =-90,keepaspectratio]{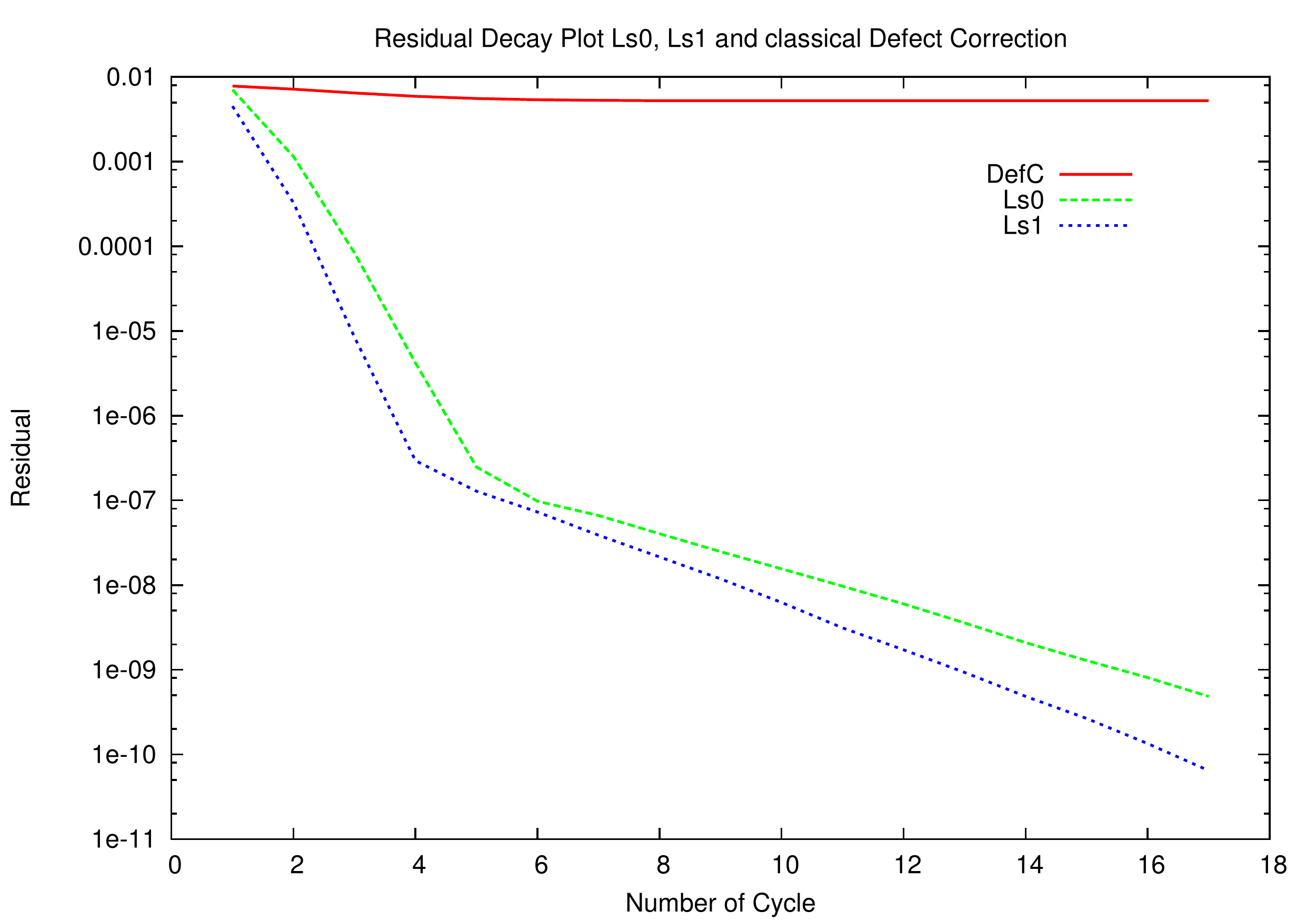}
        \caption{Comparison of residual decay of splitting $Ls0$, splitting $Ls1$ and classical Defect-correction 
        for $\kappa=0.0$ on $7^{th}$ level $V(2,1)$ cycle}
        \label{figg6}
    \end{figure}
$L^{2}$ norm error is evaluated in the following way
\begin{align}
\label{eqn96}
L^{2}(k,k-1)= \sqrt{H^{d}\sum\Big(\tilde{u}^{k-1}-I_{h}^{H}\bar{u}^{k}\Big)^{2}},
\end {align}
where $H$ is the mesh size on grid $k-1$, $\bar{u}^{k}$ is the converged solution on grid $k$ and $d$ 
denotes the dimension of the problem. The order of convergence is derived as
\begin{align}
\label{eqn97}
p_{2}=\frac{\log L^{2}(k-1,k-2)-\log L^{2}(k,k-1)}{\log 2}, 
\end{align}
where $p_{2}$ is the order of discretization in $L^{2}$ norm.
We also calculate $L^{\infty}$ and $L^1$-error and corresponding order in similar fashion.
From numerical experiments we observe that splitting $Ls0$ and $Ls1$ always show fast residual decay compare to classical 
defect-correction. Fig.~\ref{figg5} and Fig.~\ref{figg6} present the residual decay results for $Ls0$ splitting , $Ls1$ splitting 
and classical defect-correction technique for $\kappa=0.0,1/3$.
Moreover, residual decay of splitting $Ls1$ is more better than splitting $Ls0$.
On the other hand, we observe that splitting $Ls0$ has larger range of robustness
($-1.0 \le \kappa \le 0.9$) than splitting $Ls1$ ($-1.0 \le \kappa \le 0.8$). 
\vspace{\baselineskip}
\begin{table}
\caption{Comparison of $L^{\infty}$-, $L^1$-, and $L^2$-error
obtained for splitting $Ls1$ in case of the linear 
convection-diffusion equation (Example~\ref{ex:one} $\epsilon=10^{-6},\kappa=1/3$) over the domain
$\Omega=[-1, 1]\times [-1, 1]$.}
\centering 
\begin{tabular}{ccccccc}
\hline 
 $N$ &$L^{\infty}$-error & $p_{\infty}$ & $L^1$-error & $p_{1}$&$L^2$-error & $p_{2}$\\ 
\hline
16$\times$16    &1.19566e-02 & --        &2.25624e-03     &--         &1.83208e-02&--\\
32$\times$32    &2.62647e-03 & 2.1866    &3.57540e-04     &2.6577     &2.92872e-03&2.6451\\
64$\times$64    &5.70763e-04 & 2.2022    &4.33084e-05     &3.0454     &3.64904e-04&3.0047\\
128$\times$128  &1.06927e-04 & 2.4163    &5.45271e-06     &2.9896     &4.73857e-05&2.9450\\
256$\times$256  &1.92096e-05 & 2.4767    &6.79793e-07     &3.0038     &6.09179e-06&2.9595\\
512$\times$512  &3.40453e-06 & 2.4963    &8.44721e-08     &3.0085     &7.74616e-07&2.9753\\
\hline
\end{tabular}
\label{Table.l1sp1}
\end{table}
\begin{table}
\caption{Comparison of $L^{\infty}$-, $L^1$-, and $L^2$-error
obtained for splitting $Ls1$ in case of the linear 
convection-diffusion equation (Example~\ref{ex:one} $\epsilon=10^{-6},\kappa=0.0$) over the domain
$\Omega=[-1, 1]\times [-1, 1]$.}
\centering 
\begin{tabular}{ccccccc}
\hline 
 $N$ &$L_{\infty}$-error & $p_{\infty}$ & $L^1$-error & $p_{1}$&$L^2$-error & $p_{2}$\\ 
\hline
$16\times16$      & 1.27672e-02   & --         &1.49677e-03  &--        &1.36680e-02  &--   \\
$32 \times 32$    & 2.73792e-03   & 2.2213     &1.80364e-04  &3.0529    &1.82037e-03  &2.9085\\
$64 \times 64$    & 6.22587e-04   & 2.1367     &7.33006e-05  &1.2990    &6.63061e-04  &1.4570\\
$128 \times 128$  & 2.07084e-04   & 1.5881     &2.37525e-05  &1.6257    &2.13343e-04  &1.6360\\
$256 \times 256$  & 5.98623e-05   & 1.7905     &6.73718e-06  &1.8179    &5.99206e-05  &1.8321\\
$512 \times 512$  &1.58405e-05    &1.9180      &1.79203e-06  &1.9106    &1.58361e-05  &1.9198\\
\hline
\end{tabular}
\label{Table.l2sp1}
\end{table}
\begin{table}
\caption{Comparison of $L^{\infty}$-, $L^1$-, and $L^2$-error obtained for splitting
$Ls1$ in case of the linear convection-diffusion equation (Example~\ref{ex:one} $\epsilon=10^{-6},\kappa=-1.0$)
over the domain $\Omega=[-1, 1]\times [-1, 1]$.}
\centering 
\begin{tabular}{ccccccc}
\hline 
 $N$ &$L_{\infty}$-error & $p_{\infty}$ & $L^1$-error & $p_{1}$&$L^2$-error & $p_{2}$\\ 
\hline
$16 \times 16$    & 2.32470e-03 & --     & 1.56030e-02 &  --     & 2.05497e-02 &   --\\
$32\times 32$     & 1.01308e-03 & 1.1983 & 1.00995e-02 &0.62754  & 9.31777e-03 & 1.1411\\
$64\times 64$     & 3.78032e-04 & 1.4222 & 4.35094e-03 &1.2149   & 3.42296e-03 & 1.4447\\
$128 \times128$   & 1.07979e-04 & 1.8078 & 1.44691e-03 &1.5884   & 9.67504e-04 & 1.8229\\
$256\times 256$   & 2.86739e-05 & 1.9129 & 4.47319e-04 &1.6936   & 2.54980e-04 & 1.9239\\
$512 \times 512$  & 7.39007e-06 & 1.9561 & 1.28974e-04 &1.7942   & 6.53620e-05 & 1.9639\\
\hline
\end{tabular}
\label{Table.l3sp1}
\end{table}
\begin{table}
\caption{Comparison of $L^{\infty}$-, $L^1$-, and $L^2$-error obtained for splitting
$Ls0$ in case of the linear convection-diffusion equation (Example~\ref{ex:one} $\epsilon=10^{-6},\kappa=-1.0$)
over the domain $\Omega=[-1, 1]\times [-1, 1]$.}
\centering 
\begin{tabular}{ccccccc}
\hline 
 $N$ &$L_{\infty}$-error & $p_{\infty}$ & $L^1$-error & $p_{1}$&$L^2$-error & $p_{2}$\\ 
\hline
$16\times 16$     & 1.62417e-02 &--        & 2.30732e-03 & --      & 2.04151e-02 &--     \\
$32 \times 32 $   & 1.01696e-02 & 0.67544  & 1.03223e-03 & 1.1605  & 9.53668e-03 &1.0981 \\
$64 \times 64 $   & 3.89903e-03 & 1.3831   & 3.65800e-04 & 1.4966  & 3.32527e-03 &1.5200  \\
$128 \times 128$  & 1.20459e-03 & 1.6946   & 1.05973e-04 & 1.7874  & 9.49264e-04 &1.8086  \\
$256 \times 256 $ & 3.42856e-04 & 1.8129   & 2.84576e-05 & 1.8968  & 2.52429e-04 &1.9109   \\
$512 \times 512 $ & 9.05700e-05 & 1.9205   & 7.36748e-06 & 1.9496  & 6.49791e-05 &1.9578   \\
\hline
\end{tabular}
\label{Table.l4sp0}
\end{table}
\begin{table}
\caption{Comparison of $L^{\infty}$-, $L^1$-, and $L^2$-error obtained for splitting
$Ls0$ in case of the linear convection-diffusion equation (Example~\ref{ex:one} $\epsilon=10^{-6},\kappa=1/3$)
over the domain $\Omega=[-1, 1]\times [-1, 1]$.}
\centering 
\begin{tabular}{ccccccc}
\hline 
 $N$ &$L_{\infty}$-error & $p_{\infty}$ & $L^1$-error & $p_{1}$&$L^2$-error & $p_{2}$\\ 
\hline
$16   \times 16 $& 1.17579e-02 & --      & 2.25826e-03 & --      & 1.83078e-02 & --     \\
$32   \times 32 $& 1.76038e-03 & 2.7397  & 3.32640e-04 & 2.7632  & 2.72048e-03 & 2.7505  \\
$64   \times 64 $& 2.57573e-04 & 2.7728  & 4.20422e-05 & 2.9841  & 3.43166e-04 & 2.9869  \\
$128 \times 128 $& 3.47087e-05 & 2.8916  & 5.37451e-06 & 2.9676  & 4.37263e-05 & 2.9723  \\
$256 \times 256 $& 4.54820e-06 & 2.9319  & 6.78313e-07 & 2.9861  & 5.51381e-06 & 2.9874   \\
$512 \times 512 $& 6.02630e-07 & 2.9160  & 8.51091e-08 & 2.9946  & 6.91644e-07 & 2.9949   \\
\hline
\end{tabular}
\label{Table.l5sp0}
\end{table}
\begin{table}
\caption{Comparison of $L^{\infty}$-, $L^1$-, and $L^2$-error obtained for splitting
$Ls0$ in case of the linear convection-diffusion equation (Example~\ref{ex:one} $\epsilon=10^{-6},\kappa=0.0$)
over the domain $\Omega=[-1, 1]\times [-1, 1]$.}
\centering 
\begin{tabular}{ccccccc}
\hline 
$N$ &$L_{\infty}$-error & $p_{\infty}$ & $L^1$-error & $p_{1}$&$L^2$-error & $p_{2}$\\ 
\hline
$ 16 \times 16  $& 1.24738e-02 & --      & 1.50246e-03 & --      & 1.36309e-02 & --     \\
$32 \times 32   $& 2.00172e-03 & 2.6396  & 1.73122e-04 & 3.1175  & 1.68694e-03 & 3.0144 \\
$64 \times 64   $& 5.88728e-04 & 1.7656  & 6.97083e-05 & 1.3124  & 6.33640e-04 & 1.4127 \\
$128 \times 128 $& 1.84579e-04 & 1.6734  & 2.31011e-05 & 1.5934  & 2.08905e-04 & 1.6008 \\
$256 \times 256 $& 5.06126e-05 & 1.8667  & 6.64633e-06 & 1.7973  & 5.93352e-05 & 1.8159 \\
$512 \times 512 $& 1.28329e-05 & 1.9796  & 1.78059e-06 & 1.9002  & 1.57608e-05 & 1.9125 \\
\hline
\end{tabular}
\label{Table.l6sp0}
\end{table}
\subsection{Test case for numerical experiment of EHL problem}
In this section, we perform numerical experiments on EHL model defined in Section~\ref{sec:one}.
We take  Moes (\cite{moes}) dimensionless parameters (which is denoted by $M$ and $L$), where $L$ is fixed at $10$ while $M$ 
is varied between $20-1000$. For all test cases, we fix the parameter $\alpha=1.7 \times 10^{-8}$ over domain
$\Omega=[-2.5,2.5]\times[-2.5,2.5]$. In all cases , we refine grid up to $(1024+1)\times (1024+1)$ points on finest level 
and coarse grid up to $(32+1) \times (32+1)$ points on the coarsest level (except extremely high 
load case we choose coarse grid ($64+1) \times (64+1)$). A class of limiter are applied to solve the problem discussed 
in Section 3 and 4. However, for checking performance of splittings,
we use value $\kappa=0.0,1/3,-1.0$ in our numerical analysis.
In Fig.~\ref{fig:fth6}, we represent film thickness profile $\mathscr{H}$ in inverted form.
Four load cases (a)$M=20,L=10$, (b)$M=50,L=10$, (c) $M=100,L=10$ and (c) $M=1000,L=10$ are solved using the TVD schemes.
The fully converged pressure as well as film thickness profiles and their 
plot results are represented in Fig.~\ref{fig:fth6}-Fig.\ref{fig:P_M_1000_L_10}.
Comparisons of relative error in $L^{2},L^{1}$ and $L^{\infty}$ norms between
$\kappa$ splittings and defect correction schemes are performed which are presented in 
Table.~\ref{Table.ehl_defc_Hm_kp0.0}-~\ref{Table.ehl_hsp2_kp0.0_M50}.
Experimental results show that order of convergence of classical defect-correction is 
almost similar to splittings $L_{hs1}$ and $L_{hs2}$.
However, splittings $L_{hs1}$ and $L_{hs2}$ have slightly better residual
decay in comparison with classical defect-correction which can be seen in Fig.~\ref{fig:compare}.
\begin{table}
\caption{Minimum film thickness result ($M=20,L=10$) for defect-correction $\kappa=0.0$}
\centering 
\begin{tabular}{ccccccc}
\hline 
Level  & $H_{m}$       & $H_{m}$(Moes)     &$H_{c}$     &$H_{c}$(Moes)   &$H_{c}$(Moes)($p_x=0$)\\
\hline 
1   &1.99302e-01  &1.92424 &2.98940e-01  &2.88624  &2.77154\\
2   &2.59716e-01  &2.50753 &3.70695e-01  &3.57903  &3.57760\\
3   &2.70939e-01  &2.61589 &3.89566e-01  &3.76122  &3.75880\\
4   &2.74629e-01  &2.65151 &3.94288e-01  &3.80681  &3.80443\\
5   &2.75320e-01  &2.65819 &3.95428e-01  &3.81782  &3.81582\\
6   &2.75525e-01  &2.66016 &3.95886e-01  &3.82224  &3.82034\\
7   &2.75586e-01  &2.66075 &3.95962e-01  &3.82297  &3.82117\\
\hline
\end{tabular}
\label{Table.ehl_defc_Hm_kp0.0}
\end{table}
\begin{table}
\caption{Comparison of $L^{\infty}$, $L^1$ and $L^2$ relative errors
obtained with $\kappa=0.0$ by Defect-Correction over the domain
$\Omega=[-2.5,2.5]\times [-2.5,2.5]$.}
\centering 
\begin{tabular}{ccccccc}
\hline 
$N$ &$L_{\infty}$-error & $p_{\infty}$ & $L^1$-error & $p_{1}$&$L^2$-error & $p_{2}$\\ 
\hline
$ 16 \times 16  $& 1.57629e-01 & --      & 4.56501e-03 & --      & 9.85013e-02 & --     \\
$32 \times 32   $& 1.75975e-01 &-0.15884 & 2.01928e-03 & 1.1768  & 5.98804e-02 & 0.71806 \\
$64 \times 64   $& 1.69726e-01 & 0.052163& 9.26960e-04 & 1.1233  & 3.78143e-02 & 0.66315 \\
$128 \times 128 $& 1.18555e-01 & 0.51765 & 3.56082e-04 & 1.3803  & 1.79500e-02 & 1.0749 \\
$256 \times 256 $& 7.20097e-02 & 0.71929 & 1.26752e-04 & 1.4902  & 7.87096e-03 & 1.1894 \\
$512 \times 512 $& 3.16527e-02 & 1.1859  & 4.43601e-05 & 1.5147  & 2.76403e-03 & 1.5098 \\
\hline
\end{tabular}
\label{Table.ehl_defc_kp0.0}
\end{table}
\begin{table}
\caption{Comparison of $L^{\infty}$, $L^1$ and $L^2$ relative errors
obtained with $\kappa=1/3$ by Defect-Correction over the domain
$\Omega=[-2.5,2.5]\times [-2.5,2.5]$.}
\centering 
\begin{tabular}{ccccccc}
\hline 
$N$ &$L_{\infty}$-error & $p_{\infty}$ & $L^1$-error & $p_{1}$&$L^2$-error & $p_{2}$\\ 
\hline
$ 16 \times 16  $& 1.57629e-01 & --      & 4.56501e-03 & --      & 9.85013e-02 & --     \\
$32 \times 32   $& 1.75975e-01 &-0.15884 & 2.01928e-03 & 1.1768  & 5.98804e-02 & 0.71806 \\
$64 \times 64   $& 1.69726e-01 & 0.052163& 9.26960e-04 & 1.1233  & 3.78143e-02 & 0.66315 \\
$128 \times 128 $& 1.18555e-01 & 0.51765 & 3.56082e-04 & 1.3803  & 1.79500e-02 & 1.0749 \\
$256 \times 256 $& 7.20097e-02 & 0.71929 & 1.26752e-04 & 1.4902  & 7.87096e-03 & 1.1894 \\
$512 \times 512 $& 3.16527e-02 & 1.1859  & 4.43601e-05 & 1.5147  & 2.76403e-03 & 1.5098 \\
\hline
\end{tabular}
\label{Table.ehl_defc_kp0.33}
\end{table}
\begin{table}
\caption{Comparison of $L^{\infty}$, $L^1$ and $L^2$ errors
obtained with $\kappa=-1.0$ by Defect-Correction over the domain
$\Omega=[-2.5,2.5]\times [-2.5,2.5]$.}
\centering 
\begin{tabular}{ccccccc}
\hline 
$N$ &$L_{\infty}$-error & $p_{\infty}$ & $L^1$-error & $p_{1}$&$L^2$-error & $p_{2}$\\ 
\hline
$ 16 \times 16  $& 1.57629e-01 & --      & 4.56501e-03 & --      & 9.85013e-02 & --     \\
$32 \times 32   $& 1.75975e-01 &-0.15884 & 2.01928e-03 & 1.1768  & 5.98804e-02 & 0.71806 \\
$64 \times 64   $& 1.69726e-01 & 0.052163& 9.26960e-04 & 1.1233  & 3.78143e-02 & 0.66315 \\
$128 \times 128 $& 1.18555e-01 & 0.51765 & 3.56082e-04 & 1.3803  & 1.79500e-02 & 1.0749 \\
$256 \times 256 $& 7.20097e-02 & 0.71929 & 1.26752e-04 & 1.4902  & 7.87096e-03 & 1.1894 \\
$512 \times 512 $& 3.16527e-02 & 1.1859  & 4.43601e-05 & 1.5147  & 2.76403e-03 & 1.5098 \\
\hline
\end{tabular}
\label{Table.ehl_defc_kp-1.0}
\end{table}
\begin{table}
\caption{Comparison of $L^{\infty}$, $L^1$ and $L^2$ errors
obtained (M=20,L=10 case) with $\kappa=0.0$ by splitting $L_{hs1}$  over the domain
$\Omega=[-2.5,2.5]\times [-2.5,2.5]$.}
\centering 
\begin{tabular}{ccccccc}
\hline 
$N$ &$L_{\infty}$-error & $p_{\infty}$ & $L^1$-error & $p_{1}$&$L^2$-error & $p_{2}$\\ 
\hline
$ 32 \times 32    $ &7.99935e-02 &--       &3.25500e-03 &--       &4.31253e-02 &--         \\
$ 64 \times 64    $ &6.76884e-02 &0.240974 &4.20806e-04 &2.951430 &1.35161e-02 &1.673856   \\
$ 128 \times 128  $ &3.53135e-02 &0.938689 &1.14226e-04 &1.881264 &5.18955e-03 &1.380998    \\
$ 256 \times 256  $ &1.01542e-02 &1.798143 &3.02821e-05 &1.915354 &1.35755e-03 &1.934604    \\
$ 512 \times 512  $ &1.98897e-03 &2.351983 &8.51309e-06 &1.830711 &3.06834e-04 &2.145475    \\
$ 1024\times 1024 $ &4.02685e-04 &2.304298 &3.13898e-06 &1.439387 &8.16286e-05 &1.910312    \\
\hline
\end{tabular}
\label{Table.ehl_hsp1_kp0.0}
\end{table}
\begin{table}
\caption{Comparison of $L^{\infty}$, $L^1$ and $L^2$ errors
obtained for (M=20,L=10 case) with $\kappa=1/3$ by splitting $L_{hs1}$ over the domain
$\Omega=[-2.5,2.5]\times [-2.5,2.5]$.}
\centering 
\begin{tabular}{ccccccc}
\hline 
$N$ &$L_{\infty}$-error & $p_{\infty}$ & $L^1$-error & $p_{1}$&$L^2$-error & $p_{2}$\\ 
\hline
$ 32 \times  32   $ &1.28495e-01 &--       &3.46499e-03 &--       &4.97302e-02 &--              \\
$ 64 \times  64   $ &6.61681e-02 &0.957504 &4.17570e-04 &3.052761 &1.40651e-02 &1.822002        \\
$ 128 \times 128  $ &3.34724e-02 &0.983164 &1.07470e-04 &1.958084 &5.05401e-03 &1.476619        \\
$256 \times  256  $ &8.88278e-03 &1.913889 &2.70266e-05 &1.991482 &1.23452e-03 &2.033478         \\
$512 \times  512  $ &1.64936e-03 &2.429105 &7.15546e-06 &1.917264 &2.47734e-04 &2.317086         \\
$1024 \times 1024 $ &2.79280e-04 &2.562122 &2.77208e-06 &1.368076 &6.00344e-05 &2.044930         \\
\hline
\end{tabular}
\label{Table.ehl_hsp1_kp0.33}
\end{table} 
\begin{table}
\caption{Comparison of $L^{\infty}$, $L^1$ and $L^2$ errors
obtained (M=20,L=10 case) with $\kappa=-1.0$ by splitting $L_{hs1}$ over the domain
$\Omega=[-2.5,2.5]\times [-2.5,2.5]$.}
\centering 
\begin{tabular}{ccccccc}
\hline 
$N$ &$L_{\infty}$-error & $p_{\infty}$ & $L^1$-error & $p_{1}$&$L^2$-error & $p_{2}$\\ 
\hline
$ 32 \times  32   $ &7.50604e-02 &--        &2.97122e-03 &--       &4.14394e-02 &--       \\
$ 64 \times  64   $ &7.55099e-02 &-0.008614 &5.91844e-04 &2.327767 &1.69667e-02 &1.288297 \\
$ 128 \times  128 $ &4.53322e-02 &0.736130  &1.91253e-04 &1.629735 &7.61954e-03 &1.154930 \\
$ 256 \times  256 $ &1.61611e-02 &1.488011  &5.75179e-05 &1.733400 &2.50645e-03 &1.604059 \\
$ 512 \times  512 $ &4.50872e-03 &1.841736  &1.67111e-05 &1.783204 &6.94586e-04 &1.851420 \\
$1024 \times 1024 $ &1.10782e-03 &2.024994  &5.21125e-06 &1.681105 &1.89643e-04 &1.872867 \\
\hline
\end{tabular}
\label{Table.ehl_hsp1_kp-1.0}
\end{table} 
\begin{table}
\caption{Comparison of $L^{\infty}$, $L^1$ and $L^2$ errors
obtained for (M=20, L=10) with $\kappa=0.0$ by splitting $L_{hs2}$ over the domain
$\Omega=[-2.5,2.5]\times [-2.5,2.5]$.}
\centering 
\begin{tabular}{ccccccc}
\hline 
$N$ &$L_{\infty}$-error & $p_{\infty}$ & $L^1$-error & $p_{1}$&$L^2$-error & $p_{2}$\\ 
\hline
$ 32 \times  32   $ &7.91753e-02 &--       &3.24093e-03 &--       &4.29201e-02 &--        \\    
$ 64 \times  64   $ &6.76405e-02 &0.227163 &4.21527e-04 &2.942711 &1.35422e-02 &1.664191  \\
$ 128 \times  128 $ &3.53098e-02 &0.937819 &1.14185e-04 &1.884252 &5.18823e-03 &1.384148  \\
$ 256 \times  256 $ &1.01543e-02 &1.797978 &3.02794e-05 &1.914965 &1.35750e-03 &1.934290  \\
$ 512 \times  512 $ &1.99380e-03 &2.348498 &8.51277e-06 &1.830636 &3.07193e-04 &2.143735  \\
$1024 \times 1024 $ &4.04313e-04 &2.301976 &3.13219e-06 &1.442457 &8.15121e-05 &1.914059  \\
\hline
\end{tabular}
\label{Table.ehl_hsp2_kp0.0}
\end{table}  
\begin{table}
\caption{Comparison of $L^{\infty}$, $L^1$ and $L^2$ errors
obtained for (M=20, L=10) with $\kappa=1/3$ by splitting $L_{hs2}$ over the domain
$\Omega=[-2.5,2.5]\times [-2.5,2.5]$.}
\centering 
\begin{tabular}{ccccccc}
\hline 
$N$ &$L_{\infty}$-error & $p_{\infty}$ & $L^1$-error & $p_{1}$&$L^2$-error & $p_{2}$\\ 
\hline
$ 32 \times  32   $ &1.27894e-01 &--       &3.45271e-03 &--       &4.95561e-02 &--       \\
$ 64 \times  64   $ &6.61606e-02 &0.950904 &4.17669e-04 &3.047297 &1.40784e-02 &1.815579 \\
$ 128 \times 128  $ &3.34692e-02 &0.983138 &1.07437e-04 &1.958869 &5.05304e-03 &1.478260  \\
$ 256 \times 256  $ &8.88371e-03 &1.913600 &2.70267e-05 &1.991034 &1.23467e-03 &2.033026  \\
$ 512 \times  512 $ &1.65390e-03 &2.425290 &7.15902e-06 &1.916551 &2.48217e-04 &2.314452  \\
$1024 \times 1024 $ &2.80907e-04 &2.557708 &2.76808e-06 &1.370876 &5.99858e-05 &2.048909   \\
\hline
\end{tabular}
\label{Table.ehl_hsp2_kp0.33}
\end{table}  
\begin{table}
\caption{Comparison of $L^{\infty}$, $L^1$ and $L^2$ errors
obtained for (M=20, L=10) with $\kappa=-1.0$ by splitting $L_{hs2}$ over the domain
$\Omega=[-2.5,2.5]\times [-2.5,2.5]$.}
\centering 
\begin{tabular}{ccccccc}
\hline 
$N$ &$L_{\infty}$-error & $p_{\infty}$ & $L^1$-error & $p_{1}$&$L^2$-error & $p_{2}$\\ 
\hline
$ 32 \times  32   $ &7.47880e-02 &--        &2.95607e-03 &--       &4.12735e-02 &--       \\
$ 64 \times  64   $ &7.54384e-02 &-0.012492 &5.94019e-04 &2.315099 &1.70337e-02 &1.276824 \\
$ 128 \times 128  $ &4.53370e-02 &0.734610  &1.91320e-04 &1.634521 &7.62081e-03 &1.160376  \\
$ 256 \times 256  $ &1.61613e-02 &1.488146  &5.75274e-05 &1.733667 &2.50667e-03 &1.604172  \\
$ 512 \times 512  $ &4.51054e-03 &1.841171  &1.67195e-05 &1.782718 &6.94549e-04 &1.851624  \\
$1024 \times 1024 $ &1.10616e-03 &2.027740  &5.21053e-06 &1.682030 &1.89516e-04 &1.873757  \\
\hline
\end{tabular}
\label{Table.ehl_hsp2_kp-1.0}
\end{table} 
\begin{table}
\caption{Comparison of $L^{\infty}$, $L^1$ and $L^2$ errors
obtained for EHL M=50 and L=10 with $\kappa=0.0$ by splitting $L_{hs2}$ over the domain
$\Omega=[-2.5,2.5]\times [-2.5,2.5]$.}
\centering 
\begin{tabular}{ccccccc}
\hline 
$N$ &$L_{\infty}$-error & $p_{\infty}$ & $L^1$-error & $p_{1}$&$L^2$-error & $p_{2}$\\ 
\hline
$ 16 \times  16   $  &1.58602e-01 &--        &1.03810e-02 &--       &1.33934e-01 &--        \\
$ 32 \times  32   $  &1.37546e-01 &0.205497  &2.42128e-03 &2.100104 &6.26015e-02 &1.097253   \\
$ 64 \times  64   $  &9.91830e-02 &0.471749  &1.00043e-03 &1.275150 &3.00041e-02 &1.061038   \\
$ 128 \times  128 $  &1.28502e-01 &-0.373626 &5.50322e-04 &0.862272 &2.15520e-02 &0.477338   \\
$ 256 \times  256 $  &8.01042e-02 &0.681841  &3.32311e-04 &0.727742 &1.01793e-02 &1.082183    \\
$ 512 \times  512 $  &4.33380e-02 &0.886245  &9.52456e-05 &1.802810 &3.69633e-03 &1.461473    \\
\hline
\end{tabular}
\label{Table.ehl_hsp2_kp0.0_M50}
\end{table}
           \begin{figure}
        \centering
        \includegraphics[width=11cm,height=11cm,angle =-90,keepaspectratio]{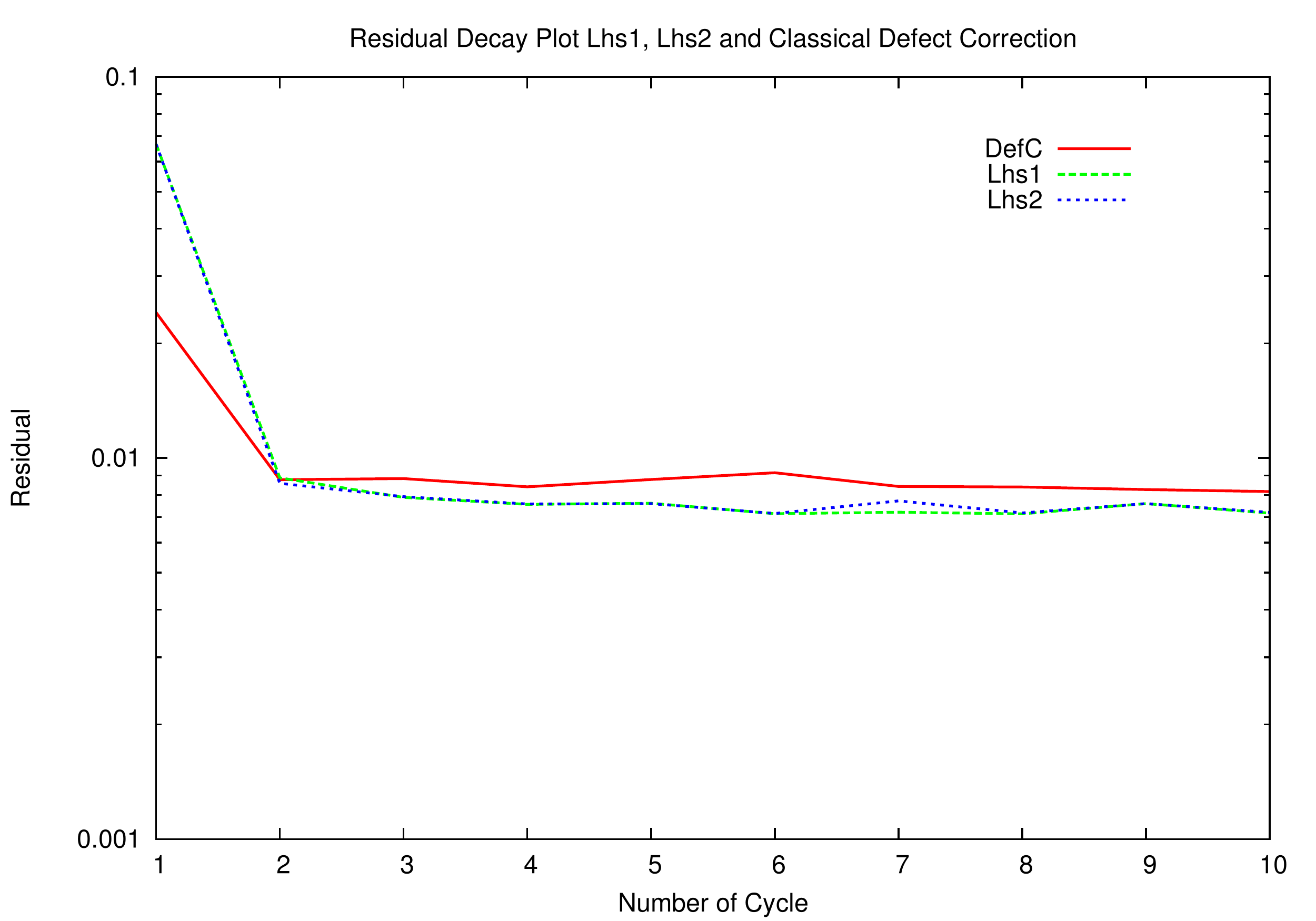}
        \caption{Comparison of residual decay of EHL by splitting $L_{hs1}$, splitting $L_{hs2}$ and classical Defect-correction 
        \label{fig:compare}
        at $\kappa=0.0$ on $7^{th}$ level $V(2,1)$ cycle}
    \end{figure}
       \begin{figure}
         \centering
         \includegraphics[width=10cm,height=12cm,keepaspectratio]{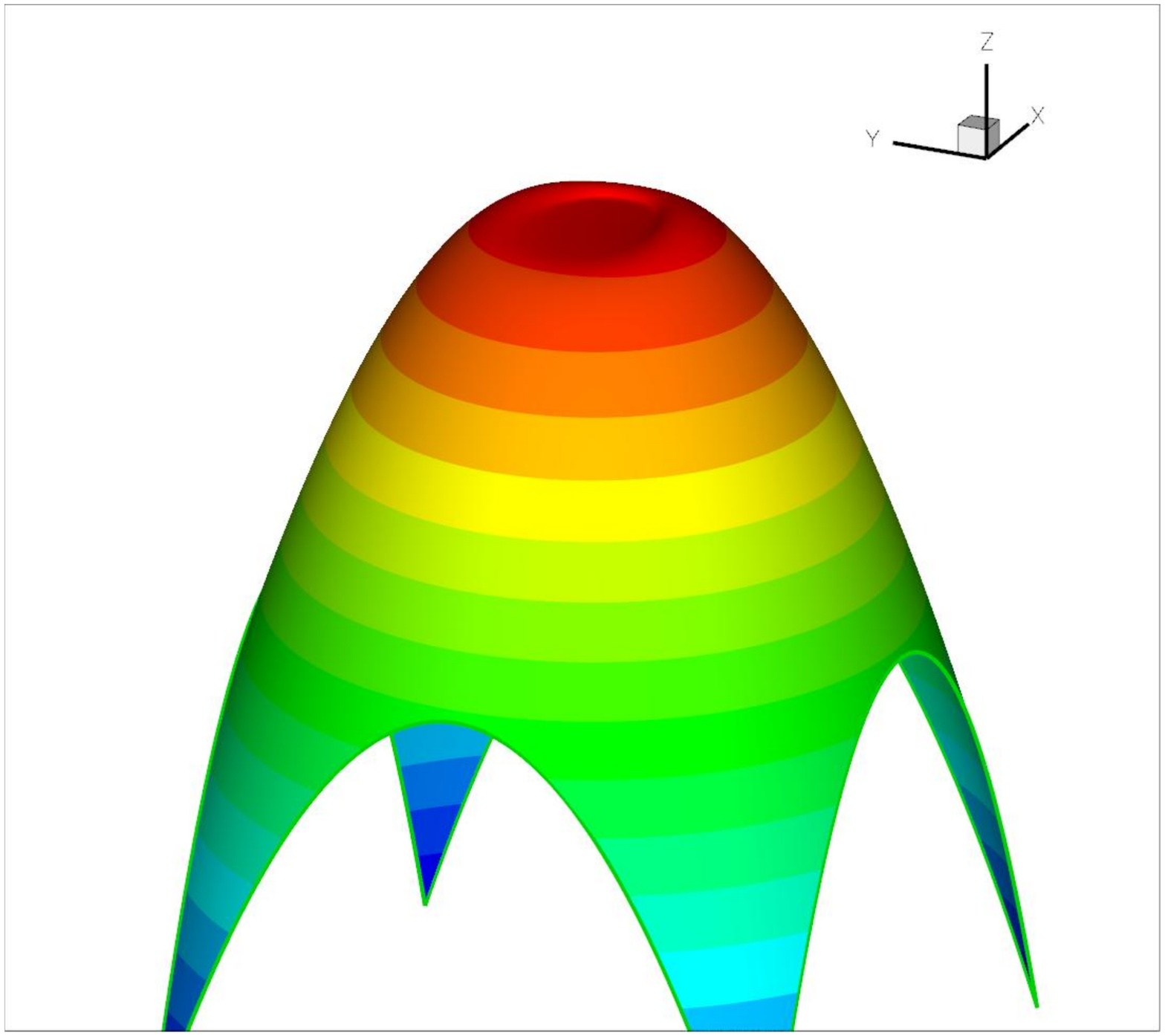}
         \caption{Typical $H$ Plot for Moes parameters $M=20 ,L=10$,$\alpha=1.7 \times 10^{-8}$ at $6^{th}$ level W-cycle}
          \label{fig:fth6}
         \end{figure}
         \begin{figure}
           \centering
           \includegraphics[width=10cm,height=12cm,keepaspectratio]{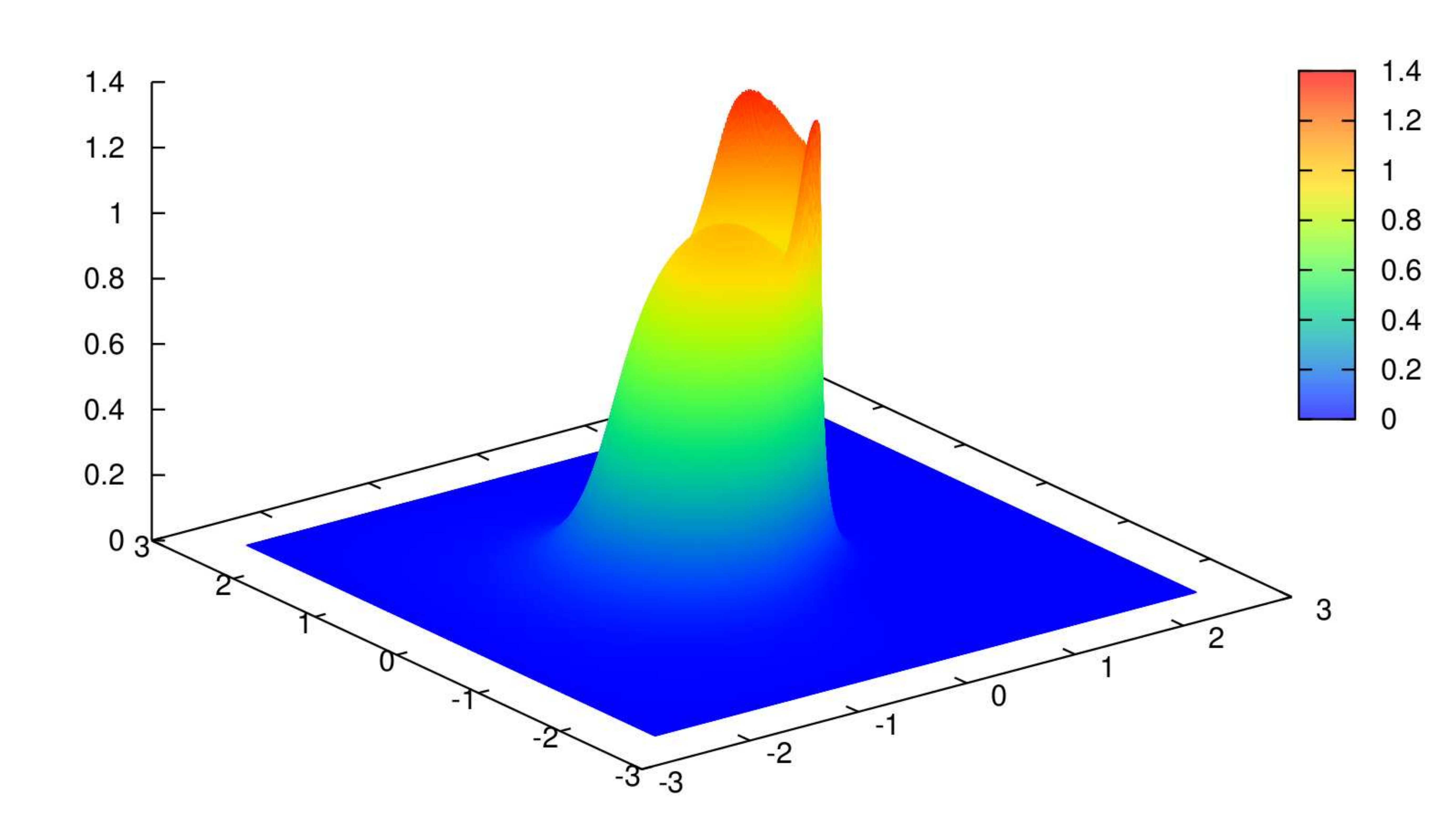}
           \caption{$P$ Plot for Moes parameters $M=20 ,L=10$,$\alpha=1.7 \times 10^{-8}$ at $6^{th}$ level W-cycle}
           \label{fig:P_M_20_L_10}
         \end{figure}
  \begin{figure}
          \centering
          \includegraphics[width=10cm,height=12cm,keepaspectratio]{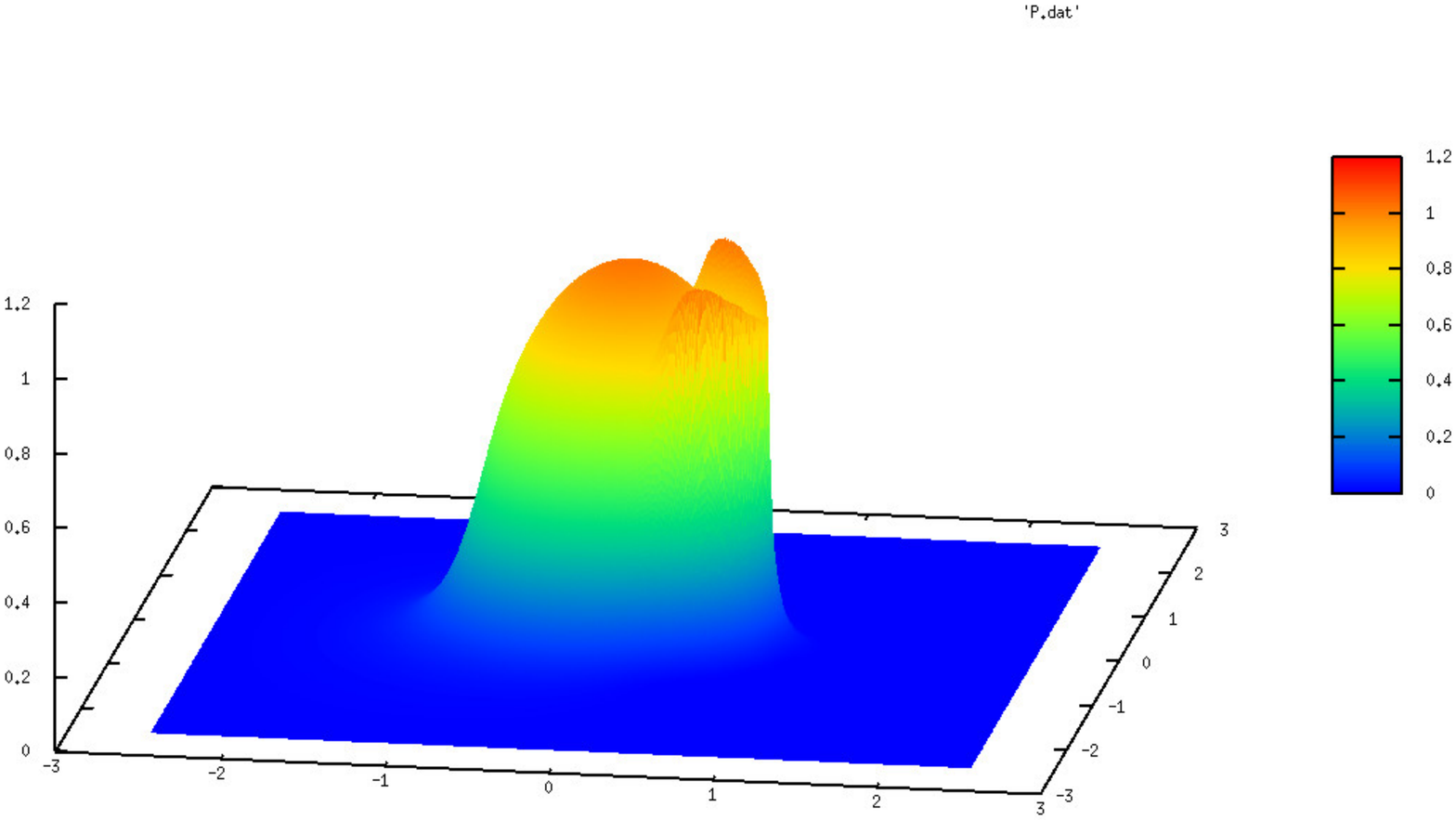}
          \caption{$P$ Plot for Moes parameters $M=50 ,L=10$,$\alpha=1.7 \times 10^{-8}$ at $7^{th}$ level V-cycle}
          \label{fig:P_M_50_L_10}
      \end{figure}
            \begin{figure}
         \centering
         \includegraphics[width=10cm,height=12cm,keepaspectratio]{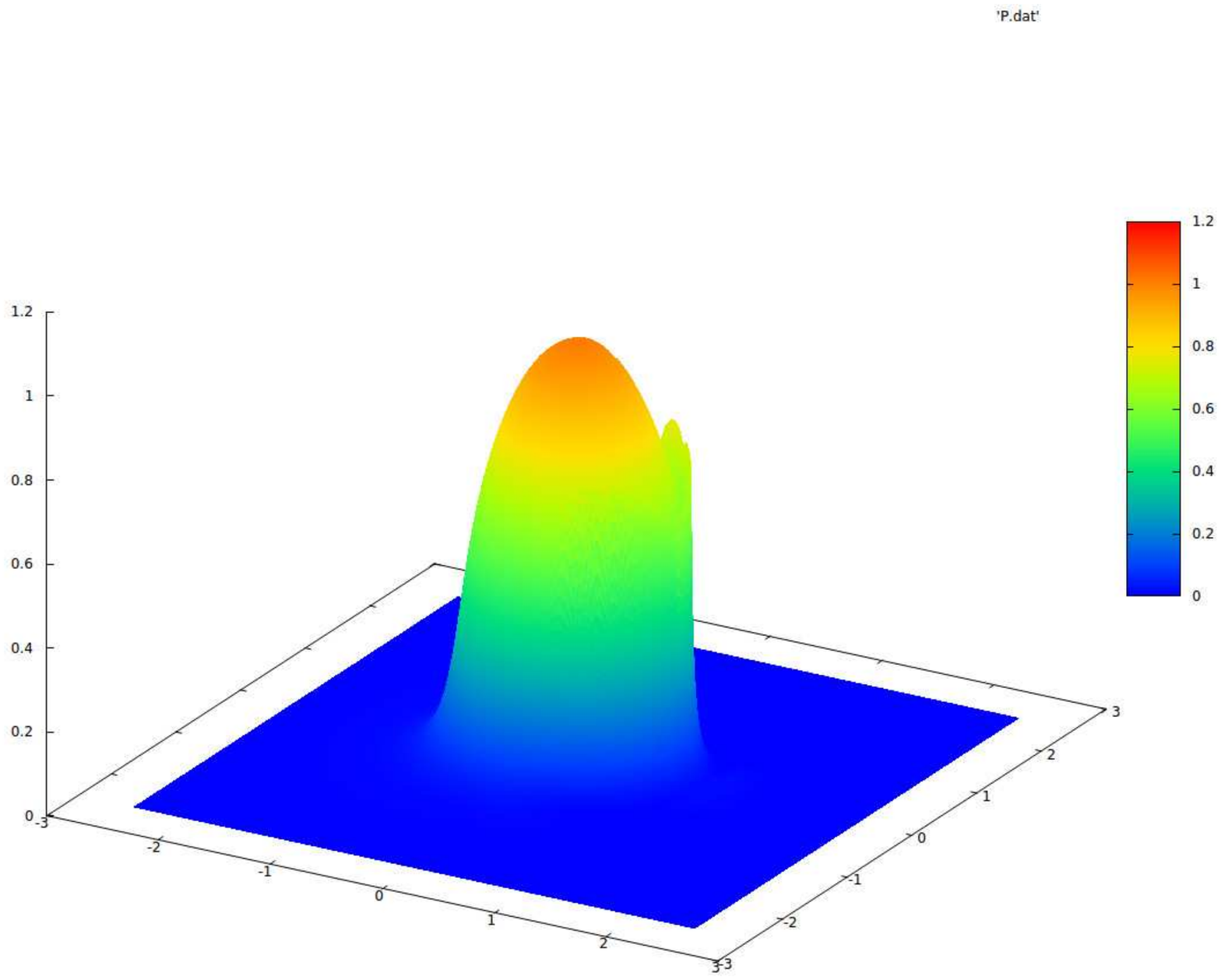}
         \caption{Pressure Plot Moes parameters $M=100 ,L=10$,$\alpha=1.7 \times 10^{-8}$ at $7^{th}$ level V-cycle}
         \label{fig:P_M_100_L_10}
     \end{figure}
            \begin{figure}
         \centering
         \includegraphics[width=10cm,height=12cm,keepaspectratio]{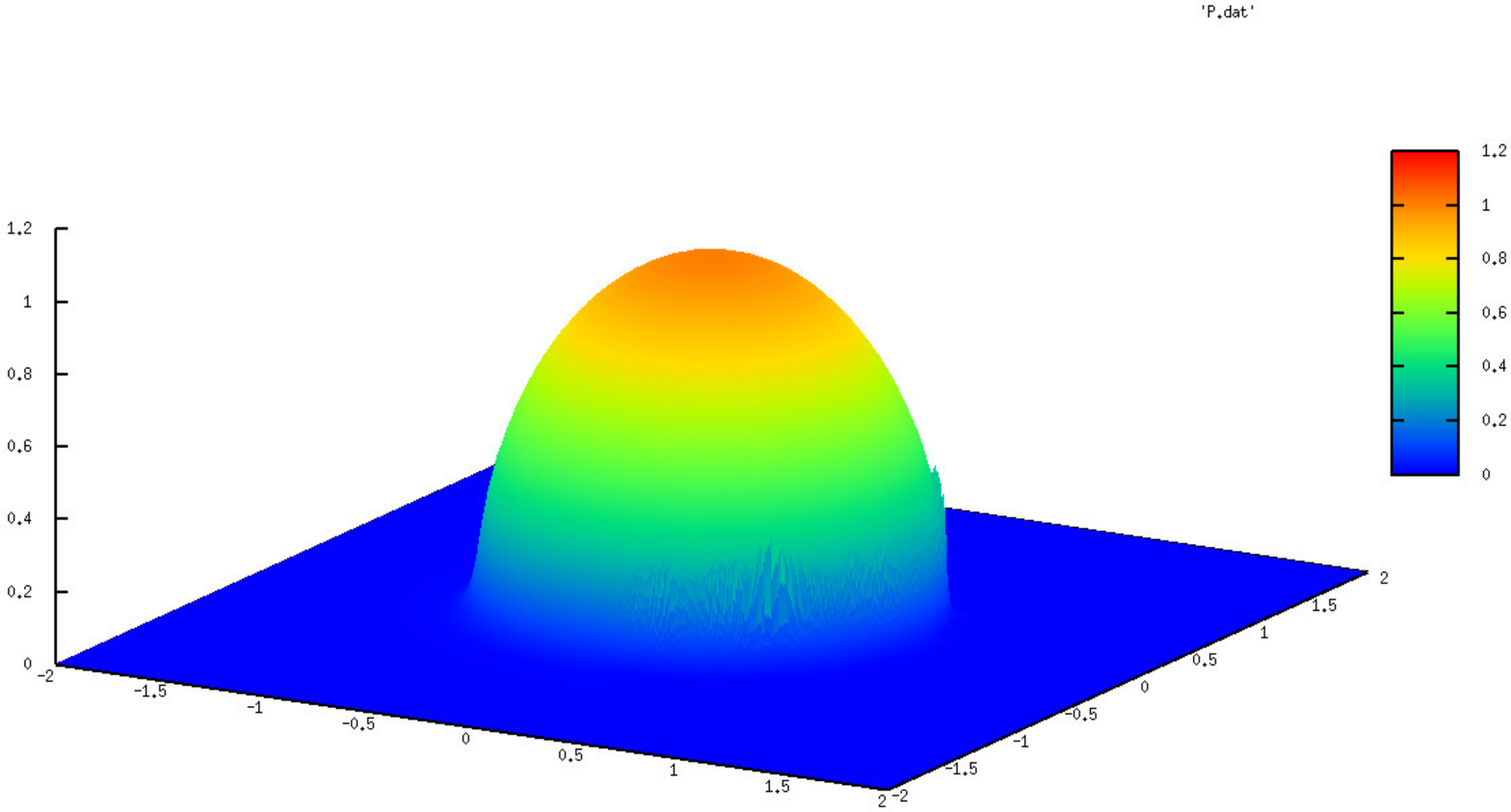}
         \caption{Pressure Plot Moes parameters $M=1000 ,L=10$,$\alpha=1.7 \times 10^{-8}$ at $7^{th}$ level V-cycle}
         \label{fig:P_M_1000_L_10}
     \end{figure}      
\section{Conclusion}\label{sec:seven}
A limiter based hybrid line splittings have been outlined for solving EHL point contact problem (in the form of LCP) on hierarchical multi level grid.
The key idea of using such splitting to facilitate artificial diffusion only the region of steep gradient of pressure profile and to
improve the accuracy on the other part (smooth region of pressure profile) of the domain.  
These illustrated splittings have been devised by bringing left hand side matrix in $M$-matrix form using 
second order discretization of Reynolds equation and rest term on the right hand side.
Additionally, the hybrid line splitting has been designed with help a switcher which depends upon magnitude of $\epsilon/h$.
When $\epsilon/h \le 0.6$, we have applied distributive Jacobi line splitting else, 
we have implemented Gauss-Seidel line splitting during updating new solution.
The derived switcher is important as it noticeably allows us in reducing the ill-conditioning of the discretized matrix when
$\epsilon$ is almost equal to zero. The robustness of the splittings have been analyzed performing series of numerical experiments.
Moreover, robustness range of splittings has been investigated and compared with other splittings. 
For linear $\kappa-$ discretization, we have performed Fourier analysis in order to validate the multi-grid convergence behavior
theoretically. Numerical experiments conform that the performance of these hybrid line splittings are 
robust not only for linear case but also for EHL model too. A remarkable achievement of these splittings are that 
it helps us in developing of higher-order discretization without losing
stability in relaxation and without the use of double discretization scheme like defect-correction technique in multi-grid solver.
Numerical experiments confirm that residual decay of direct splittings are comparably better than classical defect-correction. 
In this study, we have analyzed the performance of splittings through known limiters available in literature which works satisfactory
in all study cases. Another remarkable advantage of the adopted splittings can be noted as it does not demand 
any extra tuning parameter and produces reasonable numerical solution for large range of load variation.
\section{Acknowledgment}
This work is fully funded by DST-SERB Project reference no.PDF/2017/000202 under N-PDF fellowship program 
and working group at the Tata Institute of Fundamental Research, TIFR-CAM, Bangalore.
Author is highly indebted to IIT Kanpur for all kind of support that facilitated the
completion of this work.
\appendix
\section{Some Notation used in EHL model}\label{app:one}
$p_{H} \rightarrow$ Maximum Hertzian pressure.\\
$\eta_{0}\rightarrow$ Ambient pressure viscosity.\\
$H_{00}\rightarrow$  Central offset film thickness.\\
$a\rightarrow$ Radius of point contact circle.\\
$\alpha \rightarrow$ Pressure viscosity coefficient.\\
$u_{s} = u_{1}+u_{2}$, where $u_{1}$ upper surface velocity and $u_{2}$ lower surface velocity respectively.\\
$p_{0} \rightarrow$ Constant ($p_{0}=1.98 \times 10^{8}$), $z$ is pressure viscosity index ($z=0.68$).\\
$R \rightarrow$ Reduced radius of curvature defined as $R^{-1}=R_{1}^{-1} +R_{2}^{-1}$,\\
where $R_{1}$ and $R_{2}$ are curvature of upper contact surface and lower contact surface respectively.\\
$L$ and $M$ are Moes parameters and they are related as below.\\
$ L=G(2U)^{\frac{1}{4}}, M=W(2U)^{-\frac{1}{2}}$, where \\
$2U=\dfrac{(\eta_{0}u_{s} )}{(E^{'}R)}, W=\dfrac{F}{E'R},p_{H}=\dfrac{(3F)}{(2 \pi a^{2})}$.\\
$\sigma^{n+1}=u^{n+1}-u^{n}$ denote as difference between latest approximation solution $u^{n+1}$ and its predecessor $u^{n}$. 
\bibliographystyle{plain}
\bibliography{refer}
\end{document}